\documentclass[a4paper]{article}

\title{Rate of metastability of an iterative algorithm for quadratic optimization}

\author{Paulo Firmino$^\dagger$}

\usepackage{amsmath}
\usepackage{amssymb}
\usepackage{amsthm}
\usepackage[shortlabels]{enumitem}
\usepackage{xcolor}

\theoremstyle{plain} \newtheorem{teor}{Theorem}
\theoremstyle{plain} \newtheorem{prop}{Proposition}
\theoremstyle{definition} \newtheorem{dfn}{Definition}
\theoremstyle{plain} \newtheorem{lemma}{Lemma}
\theoremstyle{plain} \newtheorem{cor}{Corollary}
\theoremstyle{definition} \newtheorem{obs}{Remark}
\theoremstyle{plain} 
\theoremstyle{definition} 

\DeclareMathOperator{\Fix}{Fix}

\newcommand{\R}{\mathbb{R}}

\newcommand{\N}{\mathbb{N}}

\newcommand{\dd}{E}
\newcommand{\nalpha}{Q}
\newcommand{\kone}{\tilde{K}}
\newcommand{\Ns}{N_4}

\usepackage{hyperref}

\begin{document}

\maketitle

\begin{center}
	{\scriptsize $^\dagger$Department of Mathematics, Faculdade de Ciências da Universidade de Lisboa\\
	Center for Mathematical Studies, Universidade de Lisboa\\
	Campo Grande 016, 1749-016 Lisboa, Portugal\\ 
	E-mail: \protect\url{fc49883@alunos.ciencias.ulisboa.pt}\\[2mm]}
\end{center}

\begin{abstract}
In this paper, relying on methods from proof mining, we provide a quantitative analysis of a theorem due to Xu \cite{Xu03}, stating that an iteration strongly converges to the solution of a well known quadratic optimization problem. Rates of metastability and some rates of asymptotic regularity were obtained. We get quadratic rates of asymptotic regularity for particular sequences.
\end{abstract}




\section{Introduction}
Let $H$ be a real Hilbert space and $l\in\N$. Xu \cite{Xu03} considered the well known quadratic optimization problem: 
\begin{align}\label{Prob}
\mathrm{(P)} \quad\min_{x\in F} \frac{1}{2}\langle Ax,x\rangle-\langle x,u\rangle
\end{align}
where $u\in H$, $A$ is a strongly positive self-adjoint bounded linear operator and $F$ is the set of common fixed points of $T_0,\ldots,T_{l-1}$, a family of nonexpansive mappings on $H$.

Xu would proceed with an iterative approach, defining the following iteration:
\begin{equation*}
x_{n+1}=(I-\alpha_{n+1}A)T_{n+1}x_n+\alpha_{n+1}u
\end{equation*}
where $(\alpha_n) \subset (0,1]$ and we define $T_n$ for $n \ge l$ as $T_{n \bmod l}$.

Xu proved the following strong convergence result

\begin{teor} \cite{Xu03} \label{teor3.1}
Assume the following hold:
\begin{enumerate}[(i)]
\item $\lim_n \alpha_n=0$ and $\sum_{n=0}^{+\infty} \alpha_n = +\infty$,
\item $\lim_n \frac{\alpha_n}{\alpha_{n+l}}=1$ or $\sum_{n=0}^{+\infty} |\alpha_n - \alpha_{n+l}| <+\infty$,
\item $F=\Fix(T_{l-1} \cdots T_1T_0)=\Fix(T_0T_{l-1} \cdots T_1)=\ldots=\Fix(T_{l-2} \cdots T_0T_{l-1})$.
\end{enumerate}

Then $(x_n)$ converges strongly to the unique solution $x^*$ of Problem (P).
\end{teor}

Our goal is to present a quantitative analysis of Xu's result and obtain rates of metastability.

This form of quantitative analysis of mathematical results emerged from the research program of proof mining \cite{Koh08,Koh17,Koh18}, which consists in analysing mathematical proofs using proof-theoretic tools (namely, functional interpretations) with the objective of obtaining additional computational information.
Theorems regarding convergence of sequences, as this is the case, are a common interest in proof mining. In such cases, the goal is to obtain quantitative information in the form of rates of metastability and rates of asymptotic regularity. A sequence $(x_n)$ is said to be metastable if

\[\forall k\in\N \forall f\in \N^\N \exists n \forall i,j \in [n,f(n)] \left(\|x_i-x_j\|\le\frac{1}{k+1}\right)\]

A metastability rate consists of a bound $\phi(k,f)$ on $n$. Metastability is equivalent to the Cauchy property. This equivalence is ineffective. While Cauchy rates (or convergence rates) cannot be extracted in general, metastability rates can, depending on the principles used by the proof. 



Körnlein \cite{Kor, KorTh} had obtained rates of metastability of an iteration due to Yamada \cite{Yam} regarding a class of variational problems. We show a different approach, and our interest is to focus on the particular case of (P), analysing Xu's proof of Theorem \ref{teor3.1}. We note that, if $\|A\|<2$, Xu's iteration is a particular case of Yamada's iteration.


%
Xu’s proofs of Theorem \ref{teor3.1} involve sequential weak compactness arguments. In this paper we apply the method developed by Ferreira, Leu\c{s}tean and Pinto \cite{FerLeuPin} guided by the bounded functional interpretation \cite{FerOli05} for circumventing such arguments and we compute rates of metastability for Xu’s iteration.
We obtain rates of $(\tilde{T}_n)$-asymptotic regularity, where $\tilde{T}_i:= T_{i+l} \cdots T_{i+1}$ for $i\in\N$, and common rates of $T_i$-asymptotic regularity for all $i$. Quadratic rates of $(\tilde{T}_n)$-asymptotic regularity for particular sequences were obtained.








\section{Preliminaries and first results}
Let $H$ be a real Hilbert space and let $A$ be a self-adjoint bounded linear operator on $H$. More details on these operators can be found, for example, in \cite{KaAk,BauCom}. 

Let us recall that $A$ is \emph{strongly positive} if there exists $\gamma>0$ such that for all $x\in H$:
\begin{equation}\label{StrPosEq}
\langle Ax,x\rangle \ge \gamma\|x\|^2
\end{equation}

The following lemma will be used in the sequel. For the sake of completeness, we prove it here.

\begin{lemma} \label{I-minus-alpha-A-lemma}
Let $A$ be strongly positive.
For any $\alpha \in [0, \|A\|^{-1}]$:
\begin{equation}
\|I-\alpha A\|\le 1-\alpha\gamma
\end{equation}
\end{lemma}
\begin{proof}
We clearly have $\|A\| > 0$. We have that, given $\alpha \in [0, \|A\|^{-1}]$, for any $x\in H$:
\[\langle (I-\alpha A)x,x \rangle=\|x\|^2-\alpha \langle Ax,x \rangle \ge \|x\|^2-\|A\|^{-1}\|A\|\|x\|^2=0\]
Thus
\[\sup_{\|x\|=1} |\langle (I-\alpha A)x,x \rangle|=\sup_{\|x\|=1} \langle (I-\alpha A)x,x \rangle\]
As
\begin{equation*}
\sup_{\|x\|=1} |\langle Ax,x\rangle|=\|A\|
\end{equation*}
we have
\begin{equation*}
\|I-\alpha A\| = \sup_{\|x\|=1} \langle (I-\alpha A)x,x \rangle=\sup_{\|x\|=1}(\|x\|^2-\alpha \langle Ax,x \rangle) \le 1-\alpha\gamma
\end{equation*}
\end{proof}

Furthermore, we shall use the following well known inequality in real Hilbert spaces: for all $x,y\in H$,
\begin{equation} \label{lem2.4eq}
\|x+y\|^2 \le \|x\|^2+2\langle y,x+y\rangle
\end{equation}

\subsection{Quantitative definitions and lemmas}

\begin{dfn}
Let $(x_n)$ be a sequence in $H$.
\begin{enumerate}[(i)]
\item Let $x\in H$. We say $\phi:\N\to\N$ is a \emph{rate of convergence} of $(x_n)$ (to $x$) if
\[\forall k\in\N \forall n\ge\phi(k) \left(\|x_n-x\|\le\frac{1}{k+1}\right)\]
\item We say $\phi:\N\to\N$ is a \emph{Cauchy modulus} of $(x_n)$ if
\[\forall k\in\N \forall n,m\ge\phi(k) \left(\|x_n-x_m\|\le\frac{1}{k+1}\right)\]
 
\item A function $\phi:\N \times \N^\N \to \N$ is a \emph{rate of metastability} of a sequence $(x_n) \subseteq H$ if
\[\forall k\in\N \forall f\in \N^\N \exists n \le \phi(k,f) \forall i,j \in [n,f(n)] \left(\|x_i-x_j\|\le\frac{1}{k+1}\right)\]
\end{enumerate}
If $(a_n)$ is a sequence of non-negative real numbers such that $\sum_{n=0}^{+\infty} a_n=+\infty$ then $\theta:\N\to\N$ is a \emph{rate of divergence} of $(a_n)$ if $\sum_{n=0}^{\theta(k)}a_n \ge k$ for all $k\in\N$.
\begin{lemma}\label{div-rate-k-1} 
Assume that $a_n \le 1$ for all $n\in\N$ and that $\theta$ is a rate of divergence of $(a_n)$. Then $\theta(k) \ge k-1$ for all $k\in\N$.
\end{lemma}
\begin{proof}
Assume that for some $k\in\N$, we have that $\theta(k)<k-1$, hence $\theta(k)\leq k-2$.
Then $\sum\limits_{i=0}^{\theta(k)} a_i \le \sum\limits_{i=0}^{k-2} a_i \le k-1<k$, which is a contradiction.
\end{proof}

If $(b_n)$ is a sequence of real numbers such that $\limsup b_n \le 0$, then $\psi:\N\to\N$ is a \emph{$\limsup$-rate} of $(b_n)$ if
\[\forall k\in\N\forall n\ge\psi(k)\left(b_n \le \frac{1}{k+1}\right)\]
\end{dfn}

\subsubsection{Xu's Lemma}

In his paper \cite{Xu03}, Xu uses a lemma regarding the convergence of a sequence of non-negative real numbers. This result would become well known in fixed point theory and is often referred as Xu's lemma.
\begin{prop}[Xu's Lemma]\label{lem-Xu}\cite{Xu03}
Let $(a_n)$ be a sequence in $[0,1]$,
$(b_n)$ and $(c_n)$ sequences of reals and 
$(s_n)$ a sequence of nonnegative reals such that for all $n\in\N$,  
\begin{equation*}
s_{n+1}\leq (1-a_n)s_n + a_nb_n+c_n
\end{equation*}
and $\sum\limits_{n=0}^\infty a_n=+\infty,\ \limsup_n b_n \le 0,\ \sum\limits_{n=0}^\infty c_n < +\infty$.

Then $\lim_n s_n=0$.
\end{prop}
As a part of proof mining research, quantitative versions of this lemma were deduced (see \cite{Pin23} for an overview). In this paper, we use a version corresponding to the particular case of the lemma with $(c_n)\equiv 0$, whose proof can be found in \cite[Proposition 4]{LeuPin21}.

\begin{prop}\label{quant-lem-Xu02-cn-0}
Let $(a_n)$ be a sequence in $[0,1]$,
$(b_n)$ a sequence of reals and 
$(s_n)$ sequence of nonnegative reals such that for all $n\in\N$,  
\begin{equation}
s_{n+1}\leq (1-a_n)s_n + a_nb_n
\end{equation}
Assume that $L\in\N^*$ is an upper bound on $(s_n)$, $\sum\limits_{n=0}^\infty a_n$ diverges with rate of divergence $\theta$ and $\limsup_n b_n \le 0$ with $\limsup$-rate $\psi$.

Then $\lim_n s_n=0$ with rate of convergence $\Sigma_0$ defined by 
\begin{equation}
\Sigma_0(k)=\theta\left(\psi(2k+1)+\lceil \ln(2L(k+1))\rceil\right)+1.
\end{equation}
\end{prop}

The condition $\sum\limits_{n=0}^\infty a_n = \infty$ is equivalent to
\[\prod_{n=0}^\infty (1-a_n)=0.\]
As such, we may also consider this condition:
\begin{equation}\label{q2'}
\begin{gathered}
{\rm A'}:\N\times \N\to\N \text{ is a monotone function satisfying}\\
\forall k, m\in \N\, \left( \prod_{i=m}^{{\rm A'}(m,k)}(1-a_i)\leq \frac{1}{k+1}\right),
\end{gathered}
\end{equation}
implying that for each $m\in\N$, $A'(m, \cdot)$ is a rate of convergence towards zero for the sequence $\left(\prod_{i=m}^{n}(1-a_i)\right)_n$. By saying that ${\rm A'}$ is monotone we mean that it is monotone in both variables,
\[
\forall k, k', m, m' \in \N\, \left( k\leq k' \land m\leq m' \to {\rm A'}(m, k)\leq {\rm A'}(m', k')\right).
\]

We have this quantitative result involving this condition. For the proof, follow \cite[Lemma 3.4]{Pin23}, replacing $\varepsilon$ with $\frac{1}{k+1}$ (see also \cite[Lemma 13]{Pin21}).

\begin{lemma}\label{Xu-lem-C2q**}
	Let $(s_n)$ be a bounded sequence of non-negative real numbers and $D\in\N$ a positive upper bound on $(s_n)$. Consider sequences of real numbers $(a_n)\subset\, [0,1]$ and $(r_n)\subset \R$ and assume the existence of monotone functions ${\rm A'}: \N\times \N \to \N$ and ${\rm R}:\N \to \N$ such that
	\begin{enumerate}[{\rm (i)}]
		\item ${\rm A'}$ satisfies \eqref{q2'},
		\item ${\rm R}$ is such that $\forall k\in \N \, \forall n\geq {\rm R}(k) \, \left( r_n \leq \frac{1}{k+1} \right)$,
	\end{enumerate}
	 If for all $n\in \N$, $s_{n+1}\leq (1-a_n)s_n+a_nr_n$,	then $(s_n)$ converges to zero and
	\begin{equation*}
	\forall k \in \N \,\forall n\geq \Sigma^*_0(k) \, \left(s_n\leq \dfrac{1}{k+1}\right),
	\end{equation*}
	\noindent where $\Sigma^*_0(k):={\rm A'}({\rm R}(2k+1), 2D(k+1)-1)+1$.
\end{lemma}

The following is another quantitative version of Xu's lemma. This has a weaker conclusion than the previous result, in that it only gives information regarding finite terms of the sequence $(s_n)$. This is useful for obtaining rates of metastability. A proof can be seen in \cite[Lemma 2.11]{DinPin23}.
\begin{prop} \label{XuMeta}
Let $(a_n)$ be a sequence in $[0,1]$,
$(b_n), (c_n)$ sequences of real numbers and 
$(s_n)$ sequence of nonnegative real numbers such that for all $n\in\N$,  
\begin{equation}
s_{n+1}\leq (1-a_n)s_n + a_nb_n+c_n
\end{equation}
Assume that $L\in\N^*$ is an upper bound on $(s_n)$, that $\sum\limits_{n=0}^\infty a_n$ diverges with rate of 
divergence $\theta$ and that, for $k,n,q\in\N$ and for all $i\in [n,q]$:
\[c_i \le \frac{1}{3(k+1)(q+1)} \quad\text{and}\quad b_i \le \frac{1}{3(k+1)}. \]
Then
\[\forall i\in[\sigma(k,n),q]\left(s_i \le \frac{1}{k+1}\right)\]
with
\[\sigma(k,n):=\theta(n+\lceil \ln(3L(k+1))\rceil)+1.\]
\end{prop}


We also have the corresponding result with ${\rm A'}$ satisfying \eqref{q2'}. See \cite[Lemma 2.12]{DinPin23}.

\begin{prop}\label{XuMeta-C2q**}
	Let $(s_n)$ be a bounded sequence of non-negative real numbers and $D\in\N^*$ an upper bound on $(s_n)$. Consider sequences of real numbers $(a_n)\subset\, [0,1]$, $(r_n)\subset \R$ and assume that $\sum a_n=\infty$ with a function ${\rm A'}$ satisfying \eqref{q2'}. Let $k,n, q\in \N$ be given. If for all $i\in [n,q]$
	\begin{equation*}
		(i)\  s_{i+1}\leq (1-a_i)s_i+a_ib_i+c_i\,,\quad (ii)\  b_i\leq \frac{1}{3(k+1)}\,,\quad (iii)\ c_i\leq \frac{1}{3(k+1)(q+1)}, 
	\end{equation*}
	then 
	\[\forall i \in [\sigma(k,n), q] \left( s_i\leq \frac{1}{k+1}\right),\]
	where 
	\begin{equation*}
		\sigma(k,n):={\rm A'}\left(n, 3D(k+1)-1\right)+1.
	\end{equation*}
\end{prop}

\subsection{Xu's iteration}\label{iterPrelim}

Let $H$ be a Hilbert space. Given $l\in\N$, let $T_0,\ldots, T_{l-1}$ be nonexpansive mappings on $H$. We define $T_n$ for $n \ge l$ as $T_{n \bmod l}$. Define
\begin{equation}
F:= \bigcap_{i=0}^{l-1} \Fix(T_i)
\end{equation}

We assume that $F \neq \emptyset$.

Let $u\in H$. Let $A$ be a self-adjoint bounded linear operator on $H$ satisfying \eqref{StrPosEq} for some $0<\gamma \le 1$ and for all $x\in H$.

Xu's iteration $(x_n)$ is defined by
\begin{equation}\label{iter}
x_0\in H,\quad x_{n+1}=(I-\alpha_{n+1}A)T_{n+1}x_n+\alpha_{n+1}u
\end{equation}
where $(\alpha_n) \subset (0,1]$.

In the sequel we prove some useful properties of Xu's iteration.

For the rest of the section, we take $p\in F$ and assume that
\begin{center}
there is $N_0\in\N$ such that $\alpha_{n+1} \le \|A\|^{-1}$ for every $n\ge N_0$,
\end{center}
\begin{lemma}\label{bound-ineq}
For every $n\ge N_0$:
\begin{align}
\|x_{n+1}-p\| &\le(1-\gamma\alpha_{n+1})\|x_n-p\|+\alpha_{n+1}\|u-Ap\|,\label{bound-ineq1}\\
\|x_n - p\| &\leq K_0, \label{bound-ineq2}
\end{align}
where $K_0\in\N^*$ is such that:
\begin{equation} \label{def-K0}
K_0 \ge \max\left\{\|x_{N_0}-p\|,\frac{\|u-Ap\|}{\gamma}\right\}.
\end{equation}

\end{lemma}
\begin{proof}
We get that for all $n\ge N_0$:
\begin{eqnarray*}
\|x_{n+1}-p\| &=&\|(I-\alpha_{n+1}A)T_{n+1}x_n+\alpha_{n+1}u-p+\alpha_{n+1}Ap- \alpha_{n+1}Ap\|\\
&=& \|(I-\alpha_{n+1}A)(T_{n+1}x_n-p)+\alpha_{n+1}(u-Ap)\|\\
&\le& (1-\gamma\alpha_{n+1})\|T_{n+1}x_n-p\|+\alpha_{n+1}\|u-Ap\|\quad\text{by Lemma \ref{I-minus-alpha-A-lemma}}\\
&\le&(1-\gamma\alpha_{n+1})\|x_n-p\|+\alpha_{n+1}\|u-Ap\|
\end{eqnarray*}
Thus, \eqref{bound-ineq1} holds.

We prove \eqref{bound-ineq2} by induction. The case $n=N_0$ is obvious. Assume that $\|x_n-p\| \le K_0$. It follows from \eqref{bound-ineq1} that:
\begin{equation*}
\|x_{n+1}-p\| \le (1-\gamma\alpha_{n+1})K_0+\alpha_{n+1}\gamma K_0=K_0
\end{equation*}
Thus, $\|x_n-p\| \le K_0$ for every $n\ge N_0$.
\end{proof}

\begin{lemma}
Let $b_A\in\N^*$ be such that $b_A\ge \|A\|$. For every $n\ge N_0,m\in\N$ and $x\in H$:
\begin{align}
\|u-AT_{n+1}x_n\| 
&\le (b_A+1)K_0,\label{uATx-bound}\\
\begin{split}
\|x_{n+l+1}-x_{n+1}\|
&\le (1-\alpha_{n+l+1}\gamma)\|x_{n+l}-x_n\|\\
&\ \ + (b_A+1)K_0|\alpha_{n+l+1}-\alpha_{n+1}|,\label{xnl1-xn1-ineq}
\end{split}\\
\begin{split}
\|x_{n+l+m}-x_{n+m}\|
&\le \|x_{n+l}-x_n\|\prod_{i=n+1}^{n+m}(1-\alpha_{i+l}\gamma)\\
&\ \ + (b_A+1)K_0\sum_{i=n+1}^{n+m}|\alpha_{i+l}-\alpha_i|,\label{xnl1-xn1-ineq2}
\end{split}\\
\begin{split}
\|x_{n+1}-x\|^2 
&\le (1-\alpha_{n+1}\gamma)\|x_n-x\|^2+2\|x_n-x\|\|T_{n+1}x-x\|\\
&\ \ +\|T_{n+1}x-x\|^2+2\alpha_{n+1}\langle u-Ax,x_{n+1}-x\rangle.\label{xn1-x-sq-ineq}
\end{split}
\end{align}
\end{lemma}
\begin{proof}
Let $n\ge N_0$ and $x\in H$. By \eqref{bound-ineq2},
\begin{align*}
\|u-AT_{n+1}x_n\| &\le \|u-Ap\|+\|A\|\|p-T_{n+1}x_n\| \le \|u-Ap\|+\|A\|\|p-x_n\| \\
&\le \gamma K_0 +\|A\| K_0 \le (b_A+1)K_0.
\end{align*}

\begin{eqnarray*}
&&\|x_{n+l+1}-x_{n+1}\|\\
 &=& \|(I-\alpha_{n+l+1}A)T_{n+l+1}x_{n+l}+\alpha_{n+l+1}u-((I-\alpha_{n+1}A)T_{n+1}x_n+\alpha_{n+1}u)\|\\
&=& \|(I-\alpha_{n+l+1}A)(T_{n+l+1}x_{n+l}-T_{n+1}x_n)+(\alpha_{n+l+1}-\alpha_{n+1})(u-AT_{n+1}x_n)\|\\
&\le& (1-\alpha_{n+l+1}\gamma)\|T_{n+l+1}x_{n+l}-T_{n+1}x_n\|+(b_A+1)K_0|\alpha_{n+l+1}-\alpha_{n+1}|\\
&&\text{by Lemma \ref{I-minus-alpha-A-lemma} and \eqref{uATx-bound}}\\
&\le& (1-\alpha_{n+l+1}\gamma)\|x_{n+l}-x_n\|+ (b_A+1)K_0|\alpha_{n+l+1}-\alpha_{n+1}|\quad\text{since $T_{n+l}=T_n$}.
\end{eqnarray*}

By induction, from \eqref{xnl1-xn1-ineq} we get \eqref{xnl1-xn1-ineq2}.

\begin{eqnarray*}
\|x_{n+1}-x\|^2 &=& \|(I-\alpha_{n+1}A)(T_{n+1}x_n-x)+\alpha_{n+1}(u-Ax)\|^2\\
&\le& \|(I-\alpha_{n+1}A)(T_{n+1}x_n-x)\|^2+2\alpha_{n+1}\langle u-Ax,x_{n+1}-x\rangle\\
&& \text{by } \eqref{lem2.4eq}\\
&\le& (1-\alpha_{n+1}\gamma)\|T_{n+1}x_n-x\|^2+2\alpha_{n+1}\langle u-Ax,x_{n+1}-x\rangle\\
&&\text{by Lemma \ref{I-minus-alpha-A-lemma}}\\
&\le& (1-\alpha_{n+1}\gamma)(\|T_{n+1}x_n-T_{n+1}x\|+\|T_{n+1}x-x\| )^2\\
&&+2\alpha_{n+1}\langle u-Ax,x_{n+1}-x\rangle\\
&\le& (1-\alpha_{n+1}\gamma)\|x_n-x\|^2+2\|x_n-x\|\|T_{n+1}x-x\|\\
&&+\|T_{n+1}x-x\|^2+2\alpha_{n+1}\langle u-Ax,x_{n+1}-x\rangle.
\end{eqnarray*}

\end{proof}

\subsection{Quantitative hypotheses}

Xu used the following hypotheses on the parameter sequences $(\alpha_n)$:
\begin{align*}
\text{(C1)} \quad & \lim_n \alpha_n=0&\quad\text{(C2)} \quad & \sum_{n=0}^{+\infty} \alpha_n = +\infty\\
\text{(C3)} \quad & \lim_n \frac{\alpha_n}{\alpha_{n+l}}=1&\quad\text{(C4)} \quad & \sum_{n=0}^{+\infty} |\alpha_n - \alpha_{n+l}| <+\infty
\end{align*}

As we prove quantitative versions of the results in our paper, we use quantitative conditions:
\begin{eqnarray*}
\text{(C1q)} \quad & \lim_n \alpha_n=0 \text{ with convergence rate } \sigma_1\\
\text{(C2q)} \quad & \sum_{n=0}^{+\infty} \alpha_n = +\infty \text{ with divergence rate } \sigma_2\\
\text{(C2q*)} \quad &\prod_{n=m}^{+\infty} (1-\gamma\alpha_n) = 0 \text{ with rate of convergence } \sigma_2^*(m,\cdot) \text{ for } m\in\N\\
\text{(C3q)} \quad & \lim_n \frac{\alpha_n}{\alpha_{n+l}}=1 \text{ with convergence rate } \sigma_3\\
\text{(C4q)} \quad & \sum_{n=0}^{+\infty} |\alpha_n - \alpha_{n+l}| <+\infty \text{ with Cauchy modulus } \sigma_4
\end{eqnarray*}
We make natural monotonicity assumptions on the rates, i.e., that the $\sigma_i$ and $\sigma_2^*$ are increasing.

\begin{lemma}\label{lem-nalpha}
Assume that (C1q) holds. Let:
\[\nalpha:=\max\{\sigma_1(b_A-1)-1,0\}.\]
Then for every $n \ge \nalpha$ we have:
\[\alpha_{n+1} \le\|A\|^{-1}\]
\end{lemma}

\begin{proof}
Take $n \ge \nalpha$. Then:
\[\alpha_{n+1} \le \frac{1}{b_A-1+1} \le \|A\|^{-1}\]
\end{proof}

\begin{lemma}\label{bound-lem}
Assume that (C1q) holds. Let $p\in F$ and $K\in\N^*$ be such that
\begin{equation} \label{def-K}
K \ge \max\left\{\|x_{\nalpha}-p\|,\frac{\|u-Ap\|}{\gamma}\right\}.
\end{equation}
Then $\|x_n-p\| \le K$ for all $n \ge \nalpha$.
\end{lemma}
\begin{proof}
By Lemma \ref{lem-nalpha}, we have $\alpha_{n+1} \le\|A\|^{-1}$ for every $n \ge \nalpha$. We may thus apply \eqref{bound-ineq2}, taking $N_0=\nalpha$ and $K_0=K$.
\end{proof}

If (C1q) holds, we may, analogously, use any inequality from Section \ref{iterPrelim}, taking $N_0=\nalpha$ and $K_0=K$.

\section{Connection with Körnlein's results}

In the following, we show that for $\|A\|< 2$, Xu’s iteration is a particular case of an iteration due to Yamada. Körnlein \cite{Kor,KorTh} computed rates of metastability for Yamada’s iteration. Hence, for $\|A\|< 2$, these rates are also rates for Xu’s iteration.

We also show that Xu's iteration is not necessarily a particular case of Yamada's iteration if $\|A\| \ge 2$, therefore Yamada's and Körnlein's results are not applicable in this case.

The quadratic minimization problem (P) associated to the operator $A$ is given by \eqref{Prob}.
As pointed out by Xu, the problem (P) has a unique solution $x^*$ and, furthermore, (P) has the following variational characterization (\cite[Lemma 2.3.]{Xu03}):

\begin{equation} \label{VIP}
\langle u-Ax^*, y-x^*\rangle \le 0\quad\text{for all $y\in F$}.
\end{equation}

We consider the variational inequality problem (VIP) consisting in finding $x^*\in F$ such that,
\[\forall y\in F\ \langle \mathcal{F}x^*, y-x^*\rangle \ge 0\]
for some mapping $\mathcal{F}$ in $H$.

The following iteration was considered by Yamada \cite{Yam} in connection with the VIP: 
\begin{equation}\label{iterYam}
v_0\in H,\quad v_{n+1}=(1-\alpha_{n+1})T_{n+1}v_n+\alpha_{n+1}f(T_{n+1}v_n)
\end{equation}
where $f:H\to H$ is a contraction.


\begin{lemma}
If $\|I-A\|<1$, Xu's iteration is a particular case of Yamada's iteration \eqref{iterYam}.
\end{lemma}

\begin{proof}
Let $f(x):=(I-  A)x+  u$. Then, for all $x,y\in H$, $\|f(x)-f(y)\|=\|(I-  A)x+  u-(I-  A)y-  u  \|=\|(I-A)(x-y)\| \le \|I-A\|\|x-y\|$. As $\|I-A\|<1$, we have that $f$ is a contraction.
It follows that, for all $n\in\N$,
\begin{align*}
v_{n+1}&=(1-\alpha_{n+1})T_{n+1}v_n+\alpha_{n+1}((I-  A)T_{n+1}v_n+  u)\\
&=T_{n+1}v_n-\alpha_{n+1}T_{n+1}v_n+\alpha_{n+1}T_{n+1}v_n-\alpha_{n+1}  AT_{n+1}v_n+\alpha_{n+1}  u\\
&=T_{n+1}v_n-\alpha_{n+1}  (AT_{n+1}v_n-u)\\
&=(I-\alpha_{n+1}  A)T_{n+1}v_n+\alpha_{n+1} u
\end{align*}
\end{proof}

\begin{obs}

We then show that Xu's iteration is not necessarily a particular case of Yamada's iteration if $\|A\| \ge 2$.

From the above equalities we get that, if Xu's iteration is a particular case of Yamada's iteration, then $f$ is such that, for all $n\in\N$,
\[(1-\alpha_{n+1})T_{n+1}v_n+\alpha_{n+1}((I-  A)T_{n+1}v_n+  u)=(1-\alpha_{n+1})T_{n+1}v_n+\alpha_{n+1}f(T_{n+1}v_n),\]
that is,
\[f(T_{n+1}v_n)=(I-  A)T_{n+1}v_n+  u.\]
This gives
\[\|f(T_{n+1}v_n)-f(T_{n+2}v_{n+1})\|\le\|I-A\|\|T_{n+1}v_n-T_{n+2}v_{n+1}\|.\]
Now assume we can take $v_0$ such that
\[\|(I-A)(T_1v_0-T_2v_1)\|=\|I-A\|\|T_1v_0-T_2v_1\|\]
which leads to the equality for the previous inequality. Thus, if $\|I-A\|\ge 1$, $f$ may not be a contraction. This is a contradiction. We conclude that if $\|I-A\|\ge 1$, Xu's iteration is not necessarily a particular case of Yamada's iteration. This is the case if $\|A\|\ge 2$, since then $\|I-A\|\ge \|A\|-\|I\| \ge 1$.
\end{obs}

We now show how to apply Yamada's approach to the problem (P), introducing a particular case of Yamada's iteration, which turns out to be also a particular case of  Xu's iteration. 

We first note that (P) is a particular case of the problem
\[\min_{x\in F} \theta(x)\]
for some convex function $\theta: H\to \R$ twice differentiable on $H$ such that there exists $m,M>0$ such that, for all $x,v\in H$:
\[m\|v\|^2 \le \langle \theta''(x)v,v\rangle \le M\|v\|^2.\]
This, in turn, is a particular case of the VIP \cite{DeuYam}.
Deutsch and Yamada \cite{DeuYam} had introduced an iteration for this problem (also used by Xu \cite{Xu03}):
\begin{equation}\label{iterConv}
v_0\in H,\quad v_{n+1}=T_{n+1}v_n-\alpha_{n+1}\mu \theta'(T_{n+1}v_n)
\end{equation}
where $0<\mu<\frac{2}{M}$.
\begin{prop}
The iteration \eqref{iterConv} is a particular case of \eqref{iterYam}.
\end{prop}
\begin{proof}
We define $\psi(x):= \mu \theta(x)-\frac{1}{2}\|x\|^2$. Then $\psi'(x)=\mu\theta'(x)-x$. It is known \cite[Lemma 3.5]{DeuYam} that $\psi'$ is a contraction for $0<\mu<\frac{2}{M}$. We can thus take $f:=-\psi'$ in \eqref{iterYam}:
\begin{align*}
v_{n+1}&=(1-\alpha_{n+1})T_{n+1}v_n-\alpha_{n+1}\psi'(T_{n+1}v_n)\\
&=(1-\alpha_{n+1})T_{n+1}v_n-\alpha_{n+1}(\mu\theta'(T_{n+1}v_n)-T_{n+1}v_n)\\
&=(1-\alpha_{n+1})T_{n+1}v_n-\alpha_{n+1}\mu\theta'(T_{n+1}v_n)+\alpha_{n+1}T_{n+1}v_n\\
&=T_{n+1}v_n-\alpha_{n+1}\mu\theta'(T_{n+1}v_n)
\end{align*}
\end{proof}
We can thus apply Yamada's approach to the problem (P) by taking $\theta(x):=\frac{1}{2}\langle Ax,x\rangle-\langle x,u\rangle$ in \eqref{iterConv}.

\begin{prop}
Taking $\theta(x):=\frac{1}{2}\langle Ax,x\rangle-\langle x,u\rangle$ in \eqref{iterConv} we obtain:
\begin{equation}\label{iterYamQuad}
v_0\in H,\quad v_{n+1}=(I-\alpha_{n+1}\mu A)T_{n+1}v_n+\alpha_{n+1}\mu u
\end{equation}
with $0<\mu<\frac{2}{\|A\|}$.

\end{prop}
\begin{proof}
We have $\theta'(x)=Ax-u$. Thus from \eqref{iterConv} we get
\begin{align*}
v_{n+1}&=T_{n+1}v_n-\alpha_{n+1}\mu (AT_{n+1}v_n-u)\\
&=T_{n+1}v_n-\alpha_{n+1}\mu AT_{n+1}v_n+\alpha_{n+1}\mu u\\
&=(I-\alpha_{n+1}\mu A)T_{n+1}v_n+\alpha_{n+1}\mu u
\end{align*}
We take $M=\|A\|$ since $\theta''(x)=A$ and, for every $v\in H$,
\[\langle Av,v\rangle \le \|A\|\|v\|^2.\]
\end{proof}

\begin{obs}
The iteration \eqref{iterYamQuad} is identical to Xu's iteration \eqref{iter} with $\mu A$ in place of $A$ and $\mu u$ in place of $u$. That is, we take $f(x):=(I- \mu A)x+  \mu u$ in \eqref{iterYam}.

\end{obs}

If $\|A\| < 2$, then $A=\mu A_0$ for some strongly positive self-adjoint bounded linear operator $A_0$ and some $0<\mu<\frac{2}{\|A_0\|}$. 
We thus have that, if $\|A\| < 2$, Xu's iteration \eqref{iter} is a particular case of Yamada's iteration \eqref{iterYam}.

We also have that, if $\|A\| \ge 2$, Xu's iteration \eqref{iter} is not a particular case of the iteration \eqref{iterYamQuad}.




\section{Main results}

The main results of this paper are quantitative versions of Theorem \ref{teor3.1} providing uniform effective rates of metastability of Xu's iteration. These results are obtained by applying proof mining methods to Xu's strong convergence proof.

In the following, $H$ is a Hilbert space and $A$ is a self-adjoint bounded linear operator on $H$ satisfying \eqref{StrPosEq} for some $0<\gamma \le 1$ and for all $x\in H$. Furthermore, $l\in\N$ and $T_0,\ldots, T_{l-1}$ are nonexpansive mappings on $H$ such that $F:= \bigcap_{i=0}^{l-1} \Fix(T_i)$ is a nonempty set. Xu's iteration $(x_n)$ is given by \eqref{iter}, 
where $u\in H$, $(\alpha_n) \subset (0,1]$ and $T_n=T_{n \bmod l}$ for $n \ge l$.

For every $i\in\N$, denote
\[\tilde{T}_i:= T_{i+l} \cdots T_{i+1}.\]

We take $p\in F$ and $b_A,b_U,b_p\in\N^*$ are such that
$$b_A\ge\|A\|,\quad b_U\ge\|u\|,\quad b_p \ge \|p\|.$$

\begin{teor}\label{main}
Assume (C1q), (C2q), (C3q) hold. Let $K\in\N^*$ be such that
\begin{equation*}
K \ge \max\left\{\|x_{\nalpha}-p\|,\frac{\|u-Ap\|}{\gamma}\right\}, \quad \text{where}\quad \nalpha=\max\{\sigma_1(b_A-1)-1,0\}.
\end{equation*}
Suppose $\tau:\N\to\N$ is a monotone function satisfying, for all $k,m\in\N$ and $x\in \overline{B}_K(p)$: 
\begin{equation}\label{tau-cond}
\|x-\tilde{T}_mx\| \le \frac{1}{\tau(k)+1} \quad\text{implies}\quad \forall i<l \left(\|x-T_ix\|<\frac{1}{k+1}\right).
\end{equation}

Then, for every $k\in\N$ and $f:\N\to\N$:
\[\exists N \le \Omega(k,f) \forall i,j\in[N,\tilde{f}(N)] \left(\|x_i-x_j\| \le \frac{1}{k+1}\right),\]
where
\begin{equation}\label{def-meta}
\Omega(k,f)=\sigma(k_0,\tilde{\phi}(24K(h_{k,f}^{(\dd)}(0)+1)^2))+\nalpha,
\end{equation}
with

\begin{align*}
\sigma(k,n)&= \sigma_2 \left (\left \lceil \frac{n+\lceil \ln(12K^2(k+1))\rceil}{\gamma} \right\rceil +\nalpha+1\right)-\nalpha,\\
k_0&=4(k+1)^2-1,\\[2mm]
\tilde{\phi}(k)&=\Phi(\tau(k)),\\
\Phi(k)&=\max\{\Sigma((l+1)(k+1)-1),\Psi((l+1)(k+1)-1)\},\\
\Sigma(k)&=\sigma_2 \left (\left \lceil \frac{\psi(2k+1)+\lceil \ln(4K(k+1))\rceil}{\gamma} \right\rceil +\nalpha+l+1\right)-l,\\
\psi(k)&=\max\left\{\sigma_3\left(\left\lceil\frac{(b_A+1)K}{\gamma}\right\rceil(k+1)-1\right)-\nalpha-1,0\right\},\\
\Psi(k)&=\max\{\sigma_1((b_A+1)K(k+1)-1)-1,0\},\\[2mm]
\dd&=2^83^3b_A K^2\left(\left\lceil\frac{1}{2}b_A\kone^2\right\rceil +2b_U\kone\right)\left\lceil\frac{1}{\gamma}\right\rceil^2 (k+1)^4,\\
\kone&=K+b_p,\\[2mm]
h_{k,f}(m)&=\max\{\eta(k_0,\tilde{\phi}(K_m),\tilde{f}),\tilde{f}(\sigma(k_0,\tilde{\phi}(K_m))),K_m\}+1,\\
K_m&=24K(m+1)^2,\\
\eta(k,n,f)&=\max\{24K(k+1)(f(\sigma(k,n))+1),f(\sigma(k,n)),6k+5\},\\
\tilde{f}(n)&= \max_{k\le n}f(k).\\
\end{align*}
\end{teor}

The rate of metastability $\Omega$ depends solely on the bounds $K,b_A,b_U,b_p$ and the rates for the quantitative conditions $\sigma_1,\sigma_2,\sigma_3,\tau$.

The following is the corresponding main theorem with (C2q*).

\begin{teor}\label{main-C2q**}
In the hypothesis of Theorem \ref{main}, assume that (C2q*) holds instead of (C2q). Then Theorem \ref{main} holds with $\sigma$ replaced by
\[\sigma^*(k,n)=\max\left\{\sigma_2^*\left(n, \frac{1}{12K^2(k+1)}\right)-\nalpha-l,1\right\}\]
and $\Sigma$ replaced by
\[\Sigma^*(k)=\max\{\sigma_2^*(\psi(2k+1), 4K(k+1)-1)-l,\nalpha+1\}.\]
\end{teor}

In the sequel we give two main theorems where we replace (C3q) with (C4q).

\begin{teor}\label{main-C4q}
In the hypothesis of Theorem \ref{main}, assume that (C4q) holds instead of (C3q). Then Theorem \ref{main} holds with $\Sigma$ replaced by
\[\hat{\Sigma}(k)=\sigma_2\left(\Ns+l+1+\left\lceil\frac{\ln(4K(k+1))}{\gamma}\right\rceil\right)\]
with $\Ns:=\sigma_4(2(b_A+1)K(k+1)-1)$.
\end{teor}

\begin{teor}\label{main-C4q-C2q**}
In the hypothesis of Theorem \ref{main}, assume that (C2q*) holds instead of (C2q) and (C4q) holds instead of (C3q). Then Theorem \ref{main} holds with $\sigma$ replaced by
\[\sigma^*(k,n)=\max\left\{\sigma_2^*\left(n, \frac{1}{12K^2(k+1)}\right)-\nalpha-l,1\right\}\]
and $\Sigma$ replaced by
\[\hat{\Sigma}^*(k)=\sigma_2^*(\Ns+1,4K(k+1)-1)\]
where $\Ns$ is as in Theorem \ref{main-C4q}.
\end{teor}


\begin{cor}
Assume that (C1), (C2) and ((C3) or (C4)) hold, and that
\[F=\Fix(T_{l-1} \cdots T_1T_0)=\Fix(T_0T_{l-1} \cdots T_1)=\ldots=\Fix(T_{l-2} \cdots T_0T_{l-1}).\]

Then $(x_n)$ converges strongly to the unique solution $x^*$ of Problem (P).
\end{cor}


\begin{obs}
If we replace (C3) by (C4) and take $A=I$, we obtain Bauschke's theorem \cite{Bau} for the case with nonexpansive mappings on $H$. It was shown in \cite{OPX} that the theorem is also valid with (C3).
\end{obs}

Thus, applying Theorem \ref{main} with $A=I$ and $b_A=\gamma=1$, we obtain:

\begin{cor}\label{Bau-cor}
Let $(x_n)$ be defined by
\begin{equation*}
x_0\in H,\quad x_{n+1}=(1-\alpha_{n+1})T_{n+1}x_n+\alpha_{n+1}u
\end{equation*}

Assume (C1q), (C2q), (C4q) hold. Let $K\in\N^*$ be such that
\begin{equation*}
K \ge \max\left\{\|x_0-p\|,\|u-p\|\right\}
\end{equation*}
Let $\tau$ be as in Theorem \ref{main}.

Then, for every $k\in\N$ and $f:\N\to\N$:
\[\exists N \le \Omega^*(k,f) \forall i,j\in[N,\tilde{f}(N)] \left(\|x_i-x_j\| \le \frac{1}{k+1}\right),\]
where
\begin{equation}\label{def-meta-Bau}
\Omega^*(k,f)=\sigma(k_0,\tilde{\phi}(24K(h_{k,f}^{(\dd)}(0)+1)^2)),
\end{equation}
with
\begin{align*}
\sigma(k,n)&= \sigma_2 \left ( n+\lceil \ln(12K^2(k+1))\rceil  +1\right),\\
k_0&=4(k+1)^2-1,\\[2mm]
\tilde{\phi}(k)&=\Phi(\tau(k)),\\
\Phi(k)&=\max\{\hat{\Sigma}((l+1)(k+1)-1),\Psi((l+1)(k+1)-1)\},\\
\hat{\Sigma}(k)&=\sigma_2\left(\Ns+l+1+\left\lceil\ln(4K(k+1))\right\rceil\right),\\
\Ns&=\sigma_4(4K(k+1)-1),\\
\Psi(k)&=\max\{\sigma_1(2K(k+1)-1)-1,0\},\\[2mm]
\dd&=2^83^3 K^2\left(\left\lceil\frac{1}{2}\kone^2\right\rceil +2b_U\kone\right) (k+1)^4,\\
\kone&=K+b_p,\\[2mm]
h_{k,f}(m)&=\max\{\eta(k_0,\tilde{\phi}(K_m),\tilde{f}),\tilde{f}(\sigma(k_0,\tilde{\phi}(K_m))),K_m\}+1,\\
K_m&=24K(m+1)^2,\\
\eta(k,n,f)&=\max\{24K(k+1)(f(\sigma(k,n))+1),f(\sigma(k,n)),6k+5\},\\
\tilde{f}(n)&= \max_{k\le n}f(k).\\
\end{align*}
\end{cor}


Our result is incomparable with the rate from \cite{FerLeuPin} (for Bauschke's Theorem in bounded convex sets). For the most part, our result is worse, since we analysed a proof of a more general theorem. However, our result improves in $\sigma$, corresponding to Proposition \ref{step7}, compared to $\sigma'$ from \cite[Proposition 6.9]{FerLeuPin}:
\[\sigma'(k,n):=\sigma_2\left ( \tilde{n}+\lceil \ln(3K^2(k+1)^2)\rceil  +1\right)\]
with $\tilde{n}:=\max\{n,\sigma_1(6K^2(k+1)^2)\}$.


We also have the corresponding result with (C2q*):

\begin{cor}\label{Bau-C2q**}
In the hypothesis of Corollary \ref{Bau-cor}, assume that (C2q*) holds with $\gamma=1$ instead of (C2q). Then Corollary \ref{Bau-cor} holds with $\sigma$ replaced by
\[\sigma^*(k,n)=\max\left\{\sigma_2^*\left(n, \frac{1}{12K^2(k+1)}\right)-l,1\right\}\]
and $\Sigma$ replaced by
\[\hat{\Sigma}^*(k)=\sigma_2^*(\Ns+1,4K(k+1)-1)\]
where $\Ns$ is as in Corollary \ref{Bau-cor}.
\end{cor}


\section{Proofs of our main results}
In the following, we shall present the proofs of our main quantitative results, Theorems \ref{main}, \ref{main-C2q**}, \ref{main-C4q} and \ref{main-C4q-C2q**}. We use all the notations and definitions from the previous sections.

In the sequel we present the main steps of Xu's proof.

\begin{enumerate}
\item[Step 1:] $(x_n)$ is bounded.
\item[Step 2:] $(x_n)$ is $(\tilde{T}_n)$-asymptotically regular.
\item[Step 3:] $\limsup_n \langle u-Ax^*, x_n-x^*\rangle \le 0$ where $x^*$ is the unique solution of the problem (P). The proof uses \eqref{VIP} and the Demiclosedness Principle (see for example \cite[Lemma 2.2]{Xu03}). Hence weak sequential compactness is used in this step.
\item[Step 4:] $(x_n)$ converges strongly to $x^*$. This is proved by applying Step 3, \eqref{xn1-x-sq-ineq} and Xu's lemma (Proposition \ref{lem-Xu}).
\end{enumerate}

The proofs of the main Theorems will be obtained by giving quantitative versions of each of the steps of Xu's proof. 

\subsection{Rates of asymptotic regularity}

A quantitative version of Step 2 in Xu's proof consists in giving rates of $(\tilde{T}_n)$-asymptotic regularity. Additionally, as a consequence of the quantitative hypothesis \eqref{tau-cond} on $\tau$, we obtain a simultaneous rate of $T_i$-asymptotic regularity of $(x_n)$ for all $0 \le i \le l-1$.

Let $U$ be a mapping on $H$. We say $\phi:\N\to\N$ is a \emph{rate of $U$-asymptotic regularity} of a sequence $(x_n) \subseteq H$ if
\[\forall k\in\N \forall n\ge\phi(k) \left(\|x_n-Ux_n\|\le\frac{1}{k+1}\right)\]
Let $(U_n)_{n\in\N}$ be a family of mappings on $H$. We say $\phi:\N\to\N$ is a \emph{rate of $(U_n)$-asymptotic regularity} of a sequence $(x_n) \subseteq H$ if
\[\forall k\in\N \forall n\ge\phi(k) \left(\|x_n-U_nx_n\|\le\frac{1}{k+1}\right)\]


\begin{prop}\label{tildeTn-as-reg}
Let (C1q), (C2q) and (C3q) hold. Then $(x_n)$ is $(\tilde{T}_n)$-asymptotically regular with rate $\Phi$ defined in Theorem \ref{main}.
\end{prop}
\begin{proof}$\ $\\
\noindent \textbf{Claim 1:} $\Psi(k)=\max\{\sigma_1((b_A+1)K(k+1)-1)-1,0\}$ is a rate of convergence of $\|x_{n+1}-T_{n+1}x_n\|$ to $0$.\\ 
\textbf{Proof of claim:}

We note that, since $\sigma_1$ is increasing, we have that $\nalpha=\max\{\sigma_1(b_A-1)-1,0\} \le \Psi(k)$ for every $k\in\N$.

By \eqref{uATx-bound}, we have $\|u-AT_{n+1}x_n\| \le (b_A+1)K$ for $n\ge \nalpha$.

Taking $n\ge \Psi(k)$: 
\begin{multline*} 
\|x_{n+1}-T_{n+1}x_n\|=\|(I-\alpha_{n+1}A)T_{n+1}x_n+\alpha_{n+1}u-T_{n+1}x_n\|\\=\|-\alpha_{n+1}AT_{n+1}x_n+\alpha_{n+1}u\|=\alpha_{n+1}\|u-AT_{n+1}x_n\| \\\le \frac{1}{(b_A+1)K(k+1)}(b_A+1)K=\frac{1}{k+1}
\end{multline*}
that is, $\Psi$ is a rate of convergence of $\|x_{n+1}-T_{n+1}x_n\|$ to $0$. \hfill $\blacksquare$

\vspace{10pt}

\noindent \textbf{Claim 2:} $\Sigma(k)=\sigma_2 \left (\left \lceil \frac{\psi(2k+1)+\lceil \ln(4K(k+1))\rceil}{\gamma} \right\rceil +\nalpha+l+1\right)-l$ is a rate of convergence of $\|x_{n+l}-x_n\|$ to $0$.\\
\textbf{Proof of claim:} 
We have, by Lemma \ref{xnl1-xn1-ineq}, for $n\ge \nalpha$:
\begin{equation*}
\|x_{n+l+1}-x_{n+1}\| \le (1-\alpha_{n+l+1}\gamma)\|x_{n+l}-x_n\|+\alpha_{n+l+1}\gamma \frac{(b_A+1)K|\alpha_{n+l+1}-\alpha_{n+1}|}{\alpha_{n+l+1}\gamma}
\end{equation*}

We will apply Proposition \ref{quant-lem-Xu02-cn-0} with $s_n:=\|x_{n+\nalpha+l}-x_{n+\nalpha}\|,a_n:=\alpha_{n+\nalpha+l+1}\gamma$ and
\[b_n:=\frac{(b_A+1)K|\alpha_{n+\nalpha+l+1}-\alpha_{n+\nalpha+1}|}{\alpha_{n+\nalpha+l+1}\gamma}\]

We have that $\|x_{n+\nalpha+l}-x_{n+\nalpha}\| \le \|x_{n+\nalpha+l}-p\|+\|p-x_{n+\nalpha}\| \le 2K$.

By (C3q), if $n\ge\psi(k)=\max\left\{\sigma_3\left(\left\lceil\frac{(b_A+1)K}{\gamma}\right\rceil(k+1)-1\right)-\nalpha-1,0\right\}$:

\[\frac{(b_A+1)K|\alpha_{n+\nalpha+l+1}-\alpha_{n+\nalpha+1}|}{\alpha_{n+\nalpha+l+1}\gamma} \le \frac{(b_A+1)K}{\gamma}\cdot\frac{1}{\left\lceil\frac{(b_A+1)K}{\gamma}\right\rceil(k+1)-1+1} \le \frac{1}{k+1}\]
that is, $\psi$ is a convergence rate of $b_n$ to $0$, and thus a $\limsup$-rate of $(b_n)$.

Define $\Theta(k)=\max\left\{\sigma_2 \left (\left \lceil \frac{k}{\gamma} \right\rceil +\nalpha+l+1\right)-\nalpha-l-1,0\right\}$. 
Then
\begin{eqnarray*}
\sum\limits_{n=0}^{\Theta(k)} a_n & = & \gamma\sum\limits_{n=0}^{\Theta(k)} \alpha_{n+\nalpha+l+1} = \gamma\sum\limits_{n=\nalpha+l+1}^{\Theta(k)+\nalpha+l+1} \alpha_n =\gamma
\left(\sum\limits_{n=0}^{\Theta(k)+\nalpha+l+1}\alpha_n - \sum\limits_{n=0}^{\nalpha+l}\alpha_n\right)\\
& \geq &\gamma
\left(\sum\limits_{n=0}^{\sigma_2 \left (\left \lceil \frac{k}{\gamma} \right\rceil +\nalpha+l+1\right)}\alpha_n-\sum\limits_{n=0}^{\nalpha+l}\alpha_n\right)\\
& \geq & \gamma\left(\left\lceil \frac{k}{\gamma} \right\rceil +\nalpha+l+1-\sum\limits_{n=0}^{\nalpha+l}\alpha_n\right) \quad \text{by~} (C2q)\\
& \geq & \gamma\left\lceil \frac{k}{\gamma}\right\rceil \quad \text{as~}\sum\limits_{n=0}^{\nalpha+l}\alpha_n\leq \nalpha+l+1\\
& \geq & k.
\end{eqnarray*}
Thus, $\Theta$ is a rate of divergence for $\sum\limits_{n=0}^\infty a_n$.

We now apply Proposition \ref{quant-lem-Xu02-cn-0}, concluding that $\|x_{n+\nalpha+l}-x_{n+\nalpha}\| \to 0$ with rate of convergence
\begin{equation}
k \mapsto \Theta\left(\psi(2k+1)+\lceil \ln(4K(k+1))\rceil\right)+1
\end{equation}
that is, $\|x_{n+l}-x_n\| \to 0$ with rate of convergence
\begin{align*}
\Sigma(k)&=\Theta\left(\psi(2k+1)+\lceil \ln(4K(k+1))\rceil\right)+\nalpha+1\\
&=\max\left\{\sigma_2 \left (\left \lceil \frac{\psi(2k+1)+\lceil \ln(4K(k+1))\rceil}{\gamma} \right\rceil +\nalpha+l+1\right)-\nalpha-l-1,0\right\}\\
&\ \ +\nalpha+1
\end{align*}
By Lemma \ref{div-rate-k-1}, we have that
\begin{align*}
&\sigma_2 \left (\left \lceil \frac{\psi(2k+1)+\lceil \ln(4K(k+1))\rceil}{\gamma} \right\rceil +\nalpha+l+1\right)\\
\ge& \left \lceil \frac{\psi(2k+1)+\lceil \ln(4K(k+1))\rceil}{\gamma} \right\rceil +\nalpha+l\ge \nalpha+l+1
\end{align*}
Thus
\[\Sigma(k)=\sigma_2 \left (\left \lceil \frac{\psi(2k+1)+\lceil \ln(4K(k+1))\rceil}{\gamma} \right\rceil +\nalpha+l+1\right)-l\]
$\ $ \hfill $\blacksquare$
 
Taking $n\ge \Sigma((l+1)(k+1)-1)$ we get:
\[\|x_{n+l}-x_n\| \le \frac{1}{(l+1)(k+1)}\]
Taking $n\ge \Psi((l+1)(k+1)-1)$ we get:
\begin{align*}
\|x_{n+l}-T_{n+l}x_{n+l-1}\| &\le \frac{1}{(l+1)(k+1)}\\
\|T_{n+l}x_{n+l-1}-T_{n+l}T_{n+l-1}x_{n+l-2}\| &\le \frac{1}{(l+1)(k+1)}\\
&\vdots\\
\|T_{n+l}\cdots T_{n+2}x_{n+1}-T_{n+l}\cdots T_{n+1}x_n\| &\le \frac{1}{(l+1)(k+1)}
\end{align*}
Take $n\ge \Phi(k)=\max\{\Sigma((l+1)(k+1)-1),\Psi((l+1)(k+1)-1)\}$. 
Adding up the above and applying the triangle inequality, we get:
\[\|x_n-T_{n+l}\cdots T_{n+1}x_n\| \le \frac{1}{k+1}\]
that is, $\Phi$ is a rate of $(\tilde{T}_n)$-asymptotic regularity of $(x_n)$.
\end{proof}

\begin{prop}\label{C2q**-tildeTn-as-reg}
Assume that (C1q), (C2q*) and (C3q) hold. Then $(x_n)$ is $(\tilde{T}_n)$-asymptotically regular with rate $\Phi^*$, defined as $\Phi$ from Theorem \ref{main} but with $\Sigma$ replaced by
\[\Sigma^*(k)=\max\{\sigma_2^*(\psi(2k+1), 4K(k+1)-1)-l,\nalpha+1\}\]
\end{prop}

\begin{proof}
Claim 1 of Proposition \ref{tildeTn-as-reg} can be applied here.

We will apply Proposition \ref{Xu-lem-C2q**} with $s_n:=\|x_{n+\nalpha+l}-x_{n+\nalpha}\|,a_n:=\alpha_{n+\nalpha+l+1}\gamma$ and
\[b_n:=\frac{(b_A+1)K|\alpha_{n+\nalpha+l+1}-\alpha_{n+\nalpha+1}|}{\alpha_{n+\nalpha+l+1}\gamma}\]

We already have, as in Claim 2 of Proposition \ref{tildeTn-as-reg}, that \\$\psi(k)=\max\left\{\sigma_3\left(\left\lceil\frac{(b_A+1)K}{\gamma}\right\rceil(k+1)-1\right)-\nalpha-1,0\right\}$ is a $\limsup$-rate of $(b_n)$.

We define
\begin{equation}\label{def-theta*}
\Theta^*(m,k)=\max\{\sigma_2^*(m,k)-\nalpha-l-1,0\}.
\end{equation}
Then, given $m\in\N$, we have that $\Theta^*(m,\cdot)$ is a rate of convergence of $\prod\limits_{n=m}^\infty (1-a_n)$ to $0$.

We now apply Proposition \ref{Xu-lem-C2q**}, concluding that $\|x_{n+\nalpha+l}-x_{n+\nalpha}\| \to 0$ with rate of convergence
\[k\mapsto \Theta^*(\psi(2k+1), 4K(k+1)-1)+1\]
that is, $\|x_{n+l}-x_n\| \to 0$ with rate of convergence
\begin{align*}
\Sigma^*(k)&=\Theta^*(\psi(2k+1), 4K(k+1)-1)+\nalpha+1\\
&=\max\{\sigma_2^*(\psi(2k+1), 4K(k+1)-1)-\nalpha-l-1,0\}+\nalpha+1\\
&=\max\{\sigma_2^*(\psi(2k+1), 4K(k+1)-1)-l,\nalpha+1\}
\end{align*}
\end{proof}

\begin{prop}\label{C4q-tildeTn-as-reg}
Assume that (C1q), (C2q) and (C4q) hold. Then $(x_n)$ is $(\tilde{T}_n)$-asymptotically regular with rate $\tilde{\Phi}$, defined as $\Phi$ from Theorem \ref{main} but with $\Sigma$ replaced by
\[\hat{\Sigma}(k)=\sigma_2\left(\Ns+l+1+\left\lceil\frac{\ln(4K(k+1))}{\gamma}\right\rceil\right)\]
where $\Ns:=\sigma_4(2(b_A+1)K(k+1)-1)$.
\end{prop}

\begin{proof}
Claim 1 of Proposition \ref{tildeTn-as-reg} can be applied here.

We show that $\hat{\Sigma}$ is a rate of convergence of $\|x_{n+l}-x_n\|$ to $0$. The following argument is similar to \cite[Lemma 6.6.(ii)]{FerLeuPin}.

We apply \eqref{xnl1-xn1-ineq2} with $n=\Ns$. Then, noting that $1-x \le e^{-x}$ for $x \ge 0$, we get, for all $m\in\N$:
\begin{equation}\label{C4q-as-reg-eq1}
\|x_{\Ns+l+m}-x_{\Ns+m}\|
\le 2K\exp\left(-\gamma \sum_{i=\Ns+1}^{\Ns+m} \alpha_{i+l}\right)+ (b_A+1)K\sum_{i=\Ns+1}^{\Ns+m}|\alpha_{i+l}-\alpha_i|
\end{equation}

By (C4q), we have that
\begin{equation}\label{C4q-as-reg-eq2}
(b_A+1)K\sum_{i=\Ns+1}^{\Ns+m}|\alpha_{i+l}-\alpha_i| \le \frac{1}{2(k+1)}
\end{equation}

Let $M=\hat{\Sigma}(k)-\Ns=\sigma_2\left(\Ns+l+1+\left\lceil\frac{\ln(4K(k+1))}{\gamma}\right\rceil\right)-\Ns$. Note that, by Lemma \ref{div-rate-k-1}, $\sigma_2(k) \ge k-1$ for all $k\in\N$, thus $M\in\N$. By (C2q), it follows that, for all $m\ge M$:
\[\sum_{i=0}^{\Ns+m+l} \alpha_i \ge \sum_{i=0}^{\Ns+M} \alpha_i \ge \Ns+l+1+\left\lceil\frac{\ln(4K(k+1))}{\gamma}\right\rceil \ge \sum_{i=0}^{\Ns+l} \alpha_i + \left\lceil\frac{\ln(4K(k+1))}{\gamma}\right\rceil\]

Thus $\sum_{i=\Ns+1}^{\Ns+m} \alpha_{i+l} = \sum_{i=\Ns+l+1}^{\Ns+m+l} \alpha_i \ge \left\lceil\frac{\ln(4K(k+1))}{\gamma}\right\rceil$, therefore:
\begin{equation}\label{C4q-as-reg-eq3}
2K\exp\left(-\gamma \sum_{i=\Ns+1}^{\Ns+m} \alpha_{i+l}\right) \le 2K\exp(-\ln(4K(k+1)))= \frac{1}{2(k+1)}
\end{equation}
By \eqref{C4q-as-reg-eq1}, \eqref{C4q-as-reg-eq2} and \eqref{C4q-as-reg-eq3} we get that $\|x_{\Ns+l+m}-x_{\Ns+m}\| \le \frac{1}{k+1}$ for $m\ge M$, thus $\|x_{n+l}-x_n\| \le \frac{1}{k+1}$ for $n\ge \Ns+M=\hat{\Sigma}(k)$.

The proof concludes analogously to Proposition \ref{tildeTn-as-reg}.
\end{proof}

\begin{prop}\label{C2q**-C4q-tildeTn-as-reg}
Assume that (C1q), (C2q*) and (C4q) hold. Then $(x_n)$ is $(\tilde{T}_n)$-asymptotically regular with rate $\tilde{\Phi}^*$, defined as $\Phi$ from Theorem \ref{main} but with $\Sigma$ replaced by
\[\hat{\Sigma}^*(k)=\sigma_2^*(\Ns+1,4K(k+1)-1)\]
where $\Ns$ is as in Proposition \ref{C4q-tildeTn-as-reg}.
\end{prop}

\begin{proof}
Claim 1 of Proposition \ref{tildeTn-as-reg} can be applied here.

We show that $\hat{\Sigma}^*$ is a rate of convergence of $\|x_{n+l}-x_n\|$ to $0$.

Note that, by (C2q*), we have that $\sigma_2^*(m,k)\ge m$ for all $k\in\N^*$ and $m\in \N$.

We apply \eqref{xnl1-xn1-ineq2} with $n=\Ns$, obtaining for all $m\in\N$:
\begin{equation}\label{C4q-C2q**-as-reg-eq1}
\|x_{\Ns+l+m}-x_{\Ns+m}\|
\le 2K\prod_{i=\Ns+1}^{\Ns+m} (1-\alpha_{i+l}\gamma)+ (b_A+1)K\sum_{i=\Ns+1}^{\Ns+m}|\alpha_{i+l}-\alpha_i|
\end{equation}


Let $M:=\hat{\Sigma}^*(k)-\Ns=\sigma_2^*(\Ns+1,4K(k+1)-1)-\Ns \in \N$.

By (C2q*), we have, for $m\ge M$:
\begin{equation}\label{C4q-C2q**-as-reg-eq}
\prod_{i=\Ns+1}^{\Ns+m} (1-\alpha_{i+l}\gamma) \le \frac1{ 4K(k+1)}
\end{equation}

By \eqref{C4q-C2q**-as-reg-eq1},\eqref{C4q-as-reg-eq2} and \eqref{C4q-C2q**-as-reg-eq}, we get $\|x_{\Ns+l+m}-x_{\Ns+m}\| \le \frac{1}{k+1}$ for $m\ge M$, thus $\|x_{n+l}-x_n\| \le \frac{1}{k+1}$ for $n\ge \Ns+M=\hat{\Sigma}^*(k)$.
\end{proof}

\begin{teor}
Assume that $(x_n)$ is $(\tilde{T}_n)$-asymptotically regular with rate $\Phi$ such that $\Phi(k) \ge \nalpha$ for all $k\in\N$ and that $\tau:\N\to\N$ is as in Theorem \ref{main}. Then $(x_n)$ is $T_i$-asymptotically regular for all $0 \le i \le l-1$ with rate $\tilde{\phi}(k)=\Phi(\tau(k))$.
\end{teor}

\begin{proof}
Taking $n \ge \tilde{\phi}(k)=\Phi(\tau(k))$ we get:
\[\|x_n-\tilde{T}_nx_n\| \le \frac{1}{\tau(k)+1}\]
We have $n\ge \nalpha$, thus, by Lemma \ref{bound-lem}, we have that $x_n \in \overline{B}_K(p)$. Thus, by the hypothesis on $\tau$:
\begin{equation*}
\forall i<l \left(\|x_n-T_ix_n\|<\frac{1}{k+1}\right)
\end{equation*}
\end{proof}

\subsubsection{An example}

We obtain quadratic rates of $(\tilde{T}_n)$-asymptotic regularity for $(x_n)$ for a particular sequence $(\alpha_n)$.

\begin{prop}
Let $\alpha_n:=\frac{1}{\gamma(n+J)}$ with $J:=\left\lceil\frac{1}{\gamma}\right\rceil$.

Then Proposition \ref{C2q**-tildeTn-as-reg} gives us the following rate of $(\tilde{T}_n)$-asymptotic regularity for $(x_n)$:
\begin{multline*}
\Phi^*(k)=\max\{4K(l+1)(k+1)(\psi(2(l+1)(k+1)-1)+J-1)-J-l,\\
J((b_A+1)K(l+1)(k+1)-1)-1,\nalpha+1\}
\end{multline*}
with $\psi(k)=\max\left\{l\left\lceil\frac{(b_A+1)K}{\gamma}\right\rceil(k+1)-J-\nalpha-1,0\right\}$ and $\nalpha=\max\{J(b_A-1)-1,0\}$.

Proposition \ref{C2q**-C4q-tildeTn-as-reg} gives us, likewise:
\[\tilde{\Phi}^*(k)=4K(2lJ(b_A+1)K(l+1)(k+1)-1)(l+1)(k+1)-J.\]
\end{prop}

\begin{proof}
We have, for $n\in\N$, that $\gamma(n+J) \ge \gamma J\ge 1$ and thus $\alpha_n \le 1$.

In the following, take $k\in\N$.

\noindent \textbf{Claim 1:}
(C1q) holds with $\sigma_1(k)=Jk$.\\
\textbf{Proof of claim:}
We have that if $n \ge Jk$ then:
\begin{align*}
&n+\left\lceil\frac{1}{\gamma}\right\rceil\ge \left\lceil\frac{1}{\gamma}\right\rceil (k+1)\\
\Rightarrow\ &\gamma\left(n+\left\lceil\frac{1}{\gamma}\right\rceil\right)\ge k+1\\
\Rightarrow\ &\frac{1}{\gamma(n+J)} \le \frac{1}{k+1}.
\end{align*}
\hfill $\blacksquare$

\vspace{10pt}

\noindent \textbf{Claim 2:}
(C2q*) holds with $\sigma_2^*(m,k)=\max\{(m+J-1)(k+1)-J,0\}$.\\
\textbf{Proof of claim:}
\begin{equation*}
\prod_{i=m}^n \left(1-\gamma\frac{1}{\gamma(i+J)}\right)=\prod_{i=m}^n \frac{i+J-1}{i+J}=\frac{m+J-1}{m+J}\cdots \frac{n+J-1}{n+J}=\frac{m+J-1}{n+J}.
\end{equation*}
Taking $n\ge (m+J-1)(k+1)-J$, we get $\frac{m+J-1}{n+J} \le \frac{1}{k+1}$.

\hfill $\blacksquare$

\vspace{10pt}
\noindent \textbf{Claim 3:}
(C3q) holds with $\sigma_3(k)=\max\{l(k+1)-J,0\}$.\\
\textbf{Proof of claim:}
\begin{equation*}
\frac{\alpha_n}{\alpha_{n+l}} = \frac{\frac{1}{\gamma(n+J)}}{\frac{1}{\gamma(n+l+J)}}=\frac{n+l+J}{n+J}=1+\frac{l}{n+J}.
\end{equation*}
Taking $n\ge l(k+1)-J$, we get
\[\left|\frac{\alpha_n}{\alpha_{n+l}}-1\right| = \frac{l}{n+J} \le \frac{1}{k+1}.\]
\hfill $\blacksquare$

\vspace{10pt}

\noindent \textbf{Claim 4:}
(C4q) holds with $\sigma_4(k)=\max\{lJ(k+1)-J-1,0\}$.\\
\textbf{Proof of claim:}

Take $m\ge \sigma_4(k)+1$. Then:
\[
\sum_{i=m}^n |\alpha_{i+l}-\alpha_i| = \sum_{i=m}^n\left|\frac{1}{\gamma(i+l+J)}-\frac{1}{\gamma(i+J)}\right|=\frac{1}{\gamma}\sum_{i=m}^n \left(\frac{1}{i+J}-\frac{1}{i+l+J}\right).
\]
If $n\ge m+l$,
\begin{multline*}
\frac{1}{\gamma}\sum_{i=m}^n \left(\frac{1}{i+J}-\frac{1}{i+l+J}\right)=\\\frac{1}{\gamma}\left(\frac{1}{m+J}+\ldots+\frac{1}{m+l+J-1}-\frac{1}{n+J+1}-\ldots-\frac{1}{n+J+l}\right).\end{multline*}
Thus, for any $n\in\N$, we get
\[
\sum_{i=m}^n |\alpha_{i+l}-\alpha_i| \le \frac{l}{\gamma(m+J)} \le\frac{l}{\gamma(\sigma_4(k)+J+1)}\le \frac{l}{\gamma lJ(k+1)} \le\frac{1}{k+1}
.\]
\hfill $\blacksquare$

\vspace{10pt}

We thus apply Proposition \ref{C2q**-tildeTn-as-reg} and compute the rate:
\begin{align*}
\Psi(k)&=\max\{\sigma_1((b_A+1)K(k+1)-1)-1,0\}\\
&=\max\{J((b_A+1)K(k+1)-1)-1,0\}\\
&=J((b_A+1)K(k+1)-1)-1,
\end{align*}

\[\nalpha=\max\{\sigma_1(b_A-1)-1,0\}=\max\{J(b_A-1)-1,0\},\]

\begin{align*}
\psi(k)&=\max\left\{\sigma_3\left(\left\lceil\frac{(b_A+1)K}{\gamma}\right\rceil(k+1)-1\right)-\nalpha-1,0\right\}\\
&=\max\left\{\max\left\{l\left(\left\lceil\frac{(b_A+1)K}{\gamma}\right\rceil(k+1)\right)-J,0\right\}-\nalpha-1,0\right\}\\
&=\max\left\{l\left(\left\lceil\frac{(b_A+1)K}{\gamma}\right\rceil(k+1)\right)-J-\nalpha-1,0\right\},
\end{align*}
\begin{align*}
\Sigma^*(k)&=\max\{\sigma_2^*(\psi(2k+1), 4K(k+1)-1)-l,\nalpha+1\}\\
&=\max\{\max\{4K(k+1)(\psi(2k+1)+J-1)-J,0\}-l,\nalpha+1\}\\
&=\max\{4K(k+1)(\psi(2k+1)+J-1)-J-l,\nalpha+1\}.
\end{align*}

Thus:
\begin{align*}
\Phi^*(k)&=\max\{\Sigma^*((l+1)(k+1)-1),\Psi((l+1)(k+1)-1)\}\\
&=\max\{\max\{4K(l+1)(k+1)(\psi(2(l+1)(k+1)-1)+J-1)-J-l,\\
&\ \ \nalpha+1\},J((b_A+1)K(l+1)(k+1)-1)-1\}\\
&=\max\{4K(l+1)(k+1)(\psi(2(l+1)(k+1)-1)+J-1)-J-l,\\
&\ \ J((b_A+1)K(l+1)(k+1)-1)-1,\nalpha+1\}.
\end{align*}

We now apply Proposition \ref{C2q**-C4q-tildeTn-as-reg}:
\begin{align*}
\Ns&:=\sigma_4(2(b_A+1)K(k+1)-1)=\max\{2lJ(b_A+1)K(k+1)-J-1,0\}\\&=2lJ(b_A+1)K(k+1)-J-1,
\end{align*}
\begin{align*}
\hat{\Sigma}^*(k)&=\sigma_2^*(\Ns+1,4K(k+1)-1)\\
&=\max\{4K(2lJ(b_A+1)K(k+1)-1)(k+1)-J,0\}\\
&=4K(2lJ(b_A+1)K(k+1)-1)(k+1)-J.
\end{align*}
Thus:
\begin{align*}
\tilde{\Phi}^*(k)&=\max\{\hat{\Sigma}^*((l+1)(k+1)-1),\Psi((l+1)(k+1)-1)\}\\
&=\max\{4K(2lJ(b_A+1)K(l+1)(k+1)-1)(l+1)(k+1)-J,\\
&\ \ J((b_A+1)K(l+1)(k+1)-1)-1\}\\
&=4K(2lJ(b_A+1)K(l+1)(k+1)-1)(l+1)(k+1)-J,
\end{align*}
since
\begin{align*}
&(8lJ(b_A+1)K^2(l+1)(k+1)-4K)(l+1)(k+1)-J\\
 &\ge (16lJ(b_A+1)K^2-4K)(l+1)(k+1)-J\\
 &\ge (16lJ(b_A+1)-4)K^2(l+1)(k+1)-J\\
  &\ge 14lJ(b_A+1)K^2(l+1)(k+1)-J\\
   &\ge J(b_A+1)K(l+1)(k+1)-J-1.
\end{align*}
\end{proof}






\subsection{General quantitative results on finite families of mappings}

In this subsection, we present general quantitative results that will be used to get quantitative versions of the next steps in Xu's proof. These results are adaptations to our settings of similar results from \cite{FerLeuPin, Koh11}.

\subsubsection{A general principle}

The following is a slight generalization of \cite[Proposition 6.4.]{FerLeuPin} and is proved analogously. We use this principle in order to, following the previous two results, compute a metastability rate for $(x_n)$.

\begin{prop} \label{FLPres}
Let $(X,d)$ be a metric space and let $D$ be a subset of $X$. Let $(u_n)$ be a sequence in $X$ and let $P\in\N$. Let $U_0,\ldots,U_{l-1}$ be mappings from $X$ to $X$ and $\phi_0,\ldots,\phi_{l-1}$ be mappings from $X\times X$ to $\R$.

Assume there are monotone functions $\bar{\delta},\bar{\psi},\bar{\gamma},\bar{\eta},\bar{\sigma}$ such that:
\begin{enumerate}[(i)]
\item 
$\forall k \in \N \forall f\in \N^\N \exists N \le \bar{\psi}(k,\tilde{f}) \exists x\in D \left[\forall i<l \left(d(U_i(x),x)< \frac{1}{\tilde{f}(N)+1}\right)  \right.\\\left. \land\forall n \in [N+P,\tilde{f}(N)+P] \forall i<l \left(\phi_i(x,x)<\phi_i(x,u_n) + \frac{1}{k+1}\right)\right]$
\item 
$\forall k,n \in \N \forall f\in \N^\N \forall x\in D \left[\forall i<l \left(d(U_i(x),x) \le \frac{1}{\bar{\gamma}(k,n,\tilde{f})+1}\right)  \right.\\\left. \land\forall j \in [n+P,\bar{\eta}(k,n,\tilde{f})+P] \forall i<l \left(\phi_i(x,x)\le\phi_i(x,u_j) + \frac{1}{\bar{\delta}(k)+1}\right) \right.\\\left. \to\exists M \le \bar{\sigma}(k,n,\tilde{f}) \forall m \in [M,\tilde{f}(M)] \left(d(u_m,x) \le \frac{1}{k+1}\right)\right]$
\end{enumerate}
Then:
\begin{equation*}
\forall k\in \N \forall f\in \N^\N \exists M \le \phi(k,\tilde{f}) \forall m,n \in [M,\tilde{f}(M)] \left(d(u_m,u_n) \le \frac{1}{k+1}\right)
\end{equation*}
where $\phi(k,f):= \bar{\sigma}(2k+1,\bar{\psi}(\bar{\delta}(2k+1),\bar{f}),f)$ and $\bar{f}(m) := \max\{\bar{\gamma}(2k+1,m,f),\\\bar{\eta}(2k+1,m,f)\}$.
\end{prop}

\subsubsection{Quantitative projection arguments}

Let $H$ be a Hilbert space, $D$ a bounded subset of $H$ and $b\in\N^*$ an upper bound for the diameter of $D$. Let $l\in\N$ and let $T_0,\ldots,T_{l-1}$ be nonexpansive mappings on $H$.

In \cite[Proposition 3.1]{FerLeuPin}, proof mining was applied to a weakening of a projection argument (namely, the existence of a projection of a point on a bounded closed convex set) used by Browder. The resulting quantitative statement can in fact be generalized, under a certain condition for the given bounded set, to any bounded function rather than just the squared distance to a point.

\begin{prop}\label{quant-min-gen}
Assume that for any $k\in\N$ there is $x\in D$ such that
\[\forall i<l\left(\|T_ix-x\| <\frac{1}{k+1}\right)\]
Let $\phi:D\to\R$ be a function such that there are $m_\phi,M_\phi\in\N$ satisfying $-m_\phi \le \phi(x) \le M_\phi$ for $x\in D$.

Then for any $k\in\N$ and $f:\N\to\N$ there is $N\in\N$ such that $N \le (\tilde{f}+1)^{(r)}(0)$ and
\begin{multline}
\exists x\in D \left(\forall i<l\left(\|T_ix-x\| <\frac{1}{\tilde{f}(N)+1}\right) \right.\\ \left. \land \forall z\in D\left(\forall i<l\left( \|T_iz-z\| \le \frac{1}{N+1}\right)\to \phi(x)<\phi(z) + \frac{1}{k+1}\right)\right)
\end{multline}
where $r:=(m_\phi+M_\phi)(k+1)$.
\end{prop}

\begin{proof}
We prove the following: for any $k\in\N$ and $f:\N\to\N$ there is $N\in\N$ such that $N \le \tilde{f}^{(r)}(0)$ and
\begin{multline}
\exists x\in D \left(\forall i<l\left(\|T_ix-x\| \le\frac{1}{\tilde{f}(N)+1}\right) \right.\\ \left. \land \forall z\in D\left(\forall i<l\left( \|T_iz-z\| \le \frac{1}{N+1}\right)\to \phi(x)<\phi(z) + \frac{1}{k+1}\right)\right)
\end{multline}
Applying this to $f+1$ gives the proposition.

Assume the above result is not true, that is, there are $k\in\N$ and $f:\N\to\N$ such that for all $N\in\N$ with $N \le \tilde{f}^{(r)}(0)$
\begin{multline}\label{contpos}
\forall x\in D \left(\forall i<l\left(\|T_ix-x\| \le\frac{1}{\tilde{f}(N)+1}\right)  \right.\\ \left. \to\exists z\in D\left(\forall i<l\left( \|T_iz-z\| \le \frac{1}{N+1}\right)\land \phi(z)\le\phi(x) - \frac{1}{k+1}\right)\right)
\end{multline}

Note that, since $\tilde{f}$ is monotone, the sequence $(\tilde{f}^{(j)}(0))_{j\in\N}$ is monotone.

We define $x_0,\ldots,x_{r+1}\in D$ as follows:
\begin{enumerate}[(i)]
\item By hypothesis, take $x_0$ such that
\[\forall i<l\left( \|T_ix_0-x_0\| \le \frac{1}{\tilde{f}^{(r+1)}(0)+1}\right)\]
\item Now we define $x_{j+1}$ for $j\le r$. Assume $x_j$ is such that
\[\forall i<l\left( \|T_ix_j-x_j\| \le \frac{1}{\tilde{f}^{(r-j+1)}(0)+1}\right)\]
By \eqref{contpos} with $x=x_j$ and $N=\tilde{f}^{(r-j)}(0)$, we can take $z\in D$ such that
\[\forall i<l\left( \|T_iz-z\| \le \frac{1}{\tilde{f}^{(r-j)}(0)+1}\right)\land \phi(z)\le\phi(x_j) - \frac{1}{k+1}\]

Let $x_{j+1}$ be one such $z$.
\end{enumerate}

We have, thus, for all $j\le r$,
\[\phi(x_{j+1})\le\phi(x_j) - \frac{1}{k+1}\]
which implies
\[\phi(x_{r+1}) \le \phi(x_0) - \frac{r+1}{k+1} \le M_\phi -\frac{(m_\phi+M_\phi)(k+1)+1}{k+1} = -m_\phi-\frac{1}{k+1}<-m_\phi\]
This is a contradiction.
\end{proof}

We denote, for all $x,y\in H$ and $\lambda \in [0,1]$,
\[w_\lambda(x,y):=(1-\lambda)x+\lambda y.\]

The following result is an immediate extension to a finite famiy of mappings of \cite[Lemma 2.3.]{Koh11}.
\begin{lemma}  \label{conv-quant}
Assume furthermore that $D$ is convex.

For all $k\in\N$ and $y_1,y_2\in D$,
\begin{multline}
\forall j\in\{1,2\} \forall i<l\left(\|T_iy_j-y_j\| \le \frac{1}{12b(k+1)^2}\right)\\
\to \forall \lambda\in[0,1]\forall i<l\left(\|T_iw_\lambda(y_1,y_2)-w_\lambda(y_1,y_2)\| <\frac{1}{k+1}\right)
\end{multline}
\end{lemma}

The following result is similar to \cite[Corollary 3.5.]{FerLeuPin}.

\begin{prop}\label{quant-min-gamma-gen}
Assume that $D$ and $\phi: D\to\R$ satisfy the hypotheses in Proposition \ref{quant-min-gen} and that $D$ is convex.

For any $k\in \N$ and $f:\N \to\N$ there is $N\in\N$ such that $N \le 12b((\check{f}+1)^{(r)}(0)+1)^2$ with $r:=(m_\phi+M_\phi)(k+1)$ and $\check{f}(m):=\max\{\tilde{f}(12b(m+1)^2),12b(m+1)^2\}$ and there is $x\in D$ for which the following two properties hold:
\[\forall i<l \left(\|T_ix-x\| < \frac{1}{\tilde{f}(N)+1}\right)\]
\begin{multline*}
\forall y\in D \left(\forall i<l \left(\|T_iy-y\|\le\frac{1}{N+1}\right) \to \forall \lambda\in [0,1] \left(\phi(x) < \phi(w_\lambda(x,y)) + \frac{1}{k+1}\right)\right)
\end{multline*}
\end{prop}

\begin{proof}
By Proposition \ref{quant-min-gen}, there are $x\in D$ and $N'\in\N$ such that $N' \le (\check{f}+1)^{(r)}(0)$ and
\begin{multline}\label{eq-min}
\left(\forall i<l\left(\|T_ix-x\| <\frac{1}{\check{f}(N')+1}\right) \right.\\ \left. \land \forall z\in D\left(\forall i<l\left( \|T_iz-z\| \le \frac{1}{N'+1}\right)\to \phi(x)<\phi(z) + \frac{1}{k+1}\right)\right)
\end{multline}

Let $N:=12b(N'+1)^2$. Then $N\le 12b((\check{f}+1)^{(r)}(0)+1)^2$. Furthermore, $\tilde{f}(N) \le \check{f}(N')$, thus
\[\forall i<l\left(\|T_ix-x\| <\frac{1}{\tilde{f}(N)+1}\right)\]

Now let $y\in D$ such that $\forall i<l \left(\|T_iy-y\|\le\frac{1}{N+1}\right)$. Then we have $\forall i<l \left(\|T_iy-y\|<\frac{1}{12b(N'+1)^2}\right)$. Furthermore, we have:
\[\forall i<l\left(\|T_ix-x\| <\frac{1}{12b(N'+1)^2}\right)\]

By Lemma \ref{conv-quant}, we get
\[\forall i<l\left(\|T_iw_\lambda(x,y)-w_\lambda(x,y)\| <\frac{1}{N'+1}\right)\]

By \eqref{eq-min}, we get the result.
\end{proof}

\subsection{A quantitative version of the variational formulation of the problem}

In Step 3 of Xu's proof, the variational characterization \eqref{VIP} of the quadratic optimization problem (P) is used. In the following we give a quantitative version of \eqref{VIP}.

This quantitative result corresponds to the existence of a solution of (P).
\begin{prop}\label{quant-min-P}
Take $D=\overline{B}_K(p)$. For any $k\in\N$ and $f:\N\to\N$ there is $N\in\N$ such that $N \le (\tilde{f}+1)^{(r)}(0)$ and
\begin{multline*}
\exists x\in D \left(\forall i<l\left(\|T_ix-x\| <\frac{1}{\tilde{f}(N)+1}\right) \land \forall z\in D\right.\\\left. \left(\forall i<l\left( \|T_iz-z\| \le \frac{1}{N+1}\right)\to \frac{1}{2}\langle Ax,x\rangle-\langle x,u\rangle<\frac{1}{2}\langle Az,z\rangle-\langle z,u\rangle + \frac{1}{k+1}\right)\right)
\end{multline*}
where $r:=\left(\left\lceil\frac{1}{2}b_A\kone^2\right\rceil +2b_U\kone\right)(k+1)$ where $\kone:=K+b_p$ with $b_p\in\N^*$ such that $b_p \ge \|p\|$ and $b_U\in\N^*$ is such that $b_U\ge\|u\|$.
\end{prop}

\begin{proof}
We apply Proposition \ref{quant-min-gen} with $\phi(x)=\frac{1}{2}\langle Ax,x\rangle-\langle x,u\rangle$.

The required condition for $D$ is clearly satisfied.

Take $x\in D$. Note that $\|x\| \le \kone$. Thus we have:
\[\phi(x) \le \frac{1}{2}\|A\|\|x\|^2+\|x\|\|u\| \le \left\lceil\frac{1}{2}b_A\kone^2\right\rceil +b_U\kone\]
and
\[\phi(x) \ge \frac{1}{2}\gamma\|x\|^2-\|x\|\|u\| \ge -b_U\kone\]
\end{proof}

\begin{prop}\label{quant-min-gamma-P}
Take $D=\overline{B}_K(p)$.

For any $k\in \N$ and $f:\N \to\N$ there is $N\in\N$ such that $N \le 24K((\check{f}+1)^{(r)}(0)+1)^2$ with $r:=\left(\left\lceil\frac{1}{2}b_A\kone^2\right\rceil +2b_U\kone\right)(k+1)$ and $\check{f}(m):=\max\{\tilde{f}(24K(m+1)^2),24K(m+1)^2\}$ and there is $x\in D$ for which the following two properties hold:
\[\forall i<l \left(\|T_ix-x\| < \frac{1}{\tilde{f}(N)+1}\right)\]
\begin{multline*}
\forall y\in D \left(\forall i<l \left(\|T_iy-y\|\le\frac{1}{N+1}\right) \right.\to\\ \left.\forall \lambda\in [0,1] \left(\frac{1}{2}\langle Ax,x\rangle-\langle x,u\rangle < \frac{1}{2}\langle Aw_\lambda(x,y),w_\lambda(x,y)\rangle-\langle w_\lambda(x,y),u\rangle + \frac{1}{k+1}\right)\right)
\end{multline*}
\end{prop}

\begin{proof}
We apply Proposition \ref{quant-min-gamma-gen} with $\phi(x)=\frac{1}{2}\langle Ax,x\rangle-\langle x,u\rangle$, $M_\phi=\left\lceil\frac{1}{2}b_A\kone^2\right\rceil +b_U\kone, m_\phi=b_U\kone$  (see proof of Corollary \ref{quant-min-P}) and $b=2K$.
\end{proof}

The following result is similar to \cite[Lemma 2.7.]{Koh11}, but constitutes one of the places where the difference between Xu's result and Bauschke's theorem is illustrated - the role of the linear operator $A$.

\begin{prop}\label{propkey}
Let $x,y\in H$ and $b\in\N^*$ such that $\|x-y\| \le b$. Then, for any $k\in\N$:
\begin{multline}
\frac{1}{2}\langle Ax,x\rangle-\langle x,u\rangle \le \frac{1}{2}\langle Aw_{\lambda_k}(x,y),w_{\lambda_k}(x,y)\rangle-\langle w_{\lambda_k}(x,y),u\rangle + \frac{1}{3b_A b^2(k+1)^2} \\\to \langle u-Ax,y-x \rangle < \frac{1}{k+1}
\end{multline}
where $\lambda_k:=\frac{1}{2b_Ab^2(k+1)}$.

In particular, we get:
\begin{multline}
\forall \lambda \in[0,1] \left(\frac{1}{2}\langle Ax,x\rangle-\langle x,u\rangle \right.\\
\left.\le \frac{1}{2}\langle Aw_\lambda(x,y),w_\lambda(x,y)\rangle-\langle w_\lambda(x,y),u\rangle + \frac{1}{3b_A b^2(k+1)^2}\right) \\\to \langle u-Ax,y-x \rangle < \frac{1}{k+1}
\end{multline}
\end{prop}

\begin{proof}
First we observe that, since $A$ is linear:
\begin{eqnarray*}
&\frac{1}{2}\langle Aw_{\lambda_k}(x,y),w_{\lambda_k}(x,y)\rangle-\langle w_{\lambda_k}(x,y),u\rangle\\
=&\frac{1}{2}\langle A(x+\lambda_k(y-x)),x+\lambda_k(y-x)\rangle-\langle x+\lambda_k(y-x),u\rangle\\
=&\frac{1}{2}(\langle Ax,x\rangle + \lambda_k\langle Ax,y-x\rangle +\lambda_k\langle A(y-x),x+\lambda_k(y-x)\rangle)-\langle x,u\rangle-\lambda_k\langle y-x,u\rangle
\end{eqnarray*}

We now assume that
\[\frac{1}{2}\langle Aw_{\lambda_k}(x,y),w_{\lambda_k}(x,y)\rangle-\langle w_{\lambda_k}(x,y),u\rangle -\left(\frac{1}{2}\langle Ax,x\rangle-\langle x,u\rangle\right)\ge - \frac{1}{3b_A b^2(k+1)^2}\]
that is,
\begin{multline*}
\frac{\lambda_k}{2}\langle Ax,y-x\rangle+\frac{\lambda_k}{2}\langle Ay-Ax, x+\lambda_k(y-x)\rangle-\lambda_k\langle y-x,u\rangle \ge -\frac{1}{3b_A b^2(k+1)^2}
\end{multline*}
thus we get
\begin{multline*}
\frac{1}{2}\langle Ax,y-x\rangle+\frac{1}{2}\langle Ay, x\rangle-\frac{1}{2}\langle Ax, x\rangle+\frac{1}{2}\langle Ay-Ax,\lambda_k(y-x)\rangle-\langle y-x,u\rangle \\\ge -\frac{1}{3b_A b^2(k+1)^2\lambda_k}
\end{multline*}

We now note that $\frac{1}{3b_A b^2(k+1)^2\lambda_k}=\frac{2}{3(k+1)}$. We then recall that $\langle Ay,x\rangle =\langle y,Ax\rangle$ to deduce:

\begin{equation*}
\langle Ax,y-x\rangle+\frac{\lambda_k}{2}\langle A(y-x),y-x\rangle-\langle y-x,u\rangle \ge -\frac{2}{3(k+1)}
\end{equation*}
that is,
\begin{equation*}
\langle Ax-u,y-x\rangle+\frac{\lambda_k}{2}\langle A(y-x),y-x\rangle \ge -\frac{2}{3(k+1)}
\end{equation*}
then, by Cauchy-Schwarz:
\begin{equation*}
\langle Ax-u,y-x\rangle+\frac{\lambda_k}{2}b_Ab^2 \ge -\frac{2}{3(k+1)}
\end{equation*}

Now, noting that $\frac{\lambda_k}{2}b_Ab^2=\frac{1}{4(k+1)}$:
\begin{equation*}
\langle Ax-u,y-x\rangle \ge -\frac{2}{3(k+1)}-\frac{1}{4(k+1)} > -\frac{1}{k+1}
\end{equation*}
\end{proof}

We thus deduce the quantitative version of \eqref{VIP}.

\begin{prop}\label{quant-lem2.3}
Take $D=\overline{B}_K(p)$.

For any $k\in \N$ and $f:\N \to\N$ there is $N\in\N$ such that $N \le 24K((\check{f}+1)^{(R)}(0)+1)^2$ with $R:=12b_A K^2\left(\left\lceil\frac{1}{2}b_A\kone^2\right\rceil +2b_U\kone\right)(k+1)^2$ and $\check{f}$ as in Proposition \ref{quant-min-gamma-P} and there is $x\in D$ for which the following two properties hold:
\[\forall i<l \left(\|T_ix-x\| < \frac{1}{\tilde{f}(N)+1}\right)\]
\begin{equation*}
\forall y\in D \left(\forall i<l \left(\|T_iy-y\|\le\frac{1}{N+1}\right) \to \langle u-Ax,y-x \rangle < \frac{1}{k+1}\right)
\end{equation*}
\end{prop}

\begin{proof}
We apply Proposition \ref{quant-min-gamma-P} with $12b_A K^2(k+1)^2-1$ in place of $k$ and then Proposition \ref{propkey} (with $b:=2K$).
\end{proof}

\subsection{The compactness argument adapted to Xu's iteration}\label{XuProof}

Now we expose, in accordance with the paper \cite{FerLeuPin} by Ferreira, Leu\c{s}tean and Pinto, a modification of Step 3 of Xu's proof, where the argument involving weak sequential compactness is replaced by a Heine-Borel compactness principle. This modified proof illustrates how we may use the method from the aforementioned paper in order to compute a metastability rate.

\vspace{10pt}

Let $x^*$ be the unique solution of Problem (P). Take $D=\overline{B}_K(p)$.

By \eqref{VIP} we thus get, for all $k\in\N$:
\begin{equation}
\forall y \in D \left(\forall i<l (T_iy=y) \to \langle u-Ax^*, y-x^*\rangle < \frac{1}{k+1}\right)
\end{equation}
Define, for $m\in\N$:
\begin{equation}
\Omega_m:=\bigcup_{i<l}\left\{y\in H: \|T_iy-y\| > \frac{1}{m+1}\right\}\cup\left\{y\in H: \langle u-Ax^*, y-x^*\rangle < \frac{1}{k+1}\right\}
\end{equation}
Then $D \subseteq \bigcup_{m\in\N} \Omega_m$. Then by the principle of Heine-Borel compactness, we have an $N\in\N$ such that $D \subseteq \Omega_N$ (note how the sequence $(\Omega_m)$ is increasing). Thus:
\begin{equation} \label{a}
\forall y \in D \left(\forall i<l\left( \|T_iy-y\| \le \frac{1}{N+1}\right) \to \langle u-Ax^*, y-x^*\rangle < \frac{1}{k+1}\right)
\end{equation}

The above reasoning is made quantitative by Proposition \ref{quant-lem2.3}.

As $(x_n)$ is $T_i$-asymptotically regular for $0 \le i \le l-1$, we apply \eqref{a}, deducing that $\limsup_n \langle u-Ax^*, x_n-x^*\rangle \le 0$.

The quantitative result that expresses this modification is given in the sequel.

\begin{prop}\label{cond-innerprod-xn}
Assume (C1q) holds. Take $D=\overline{B}_K(p)$. Let $\tilde{\phi}$ be a monotone simultaneous rate of $T_i$-asymptotic regularity of $(x_n)$ for all $0\le i \le l-1$ such that $\tilde{\phi}(k) \ge \nalpha$ for all $k\in\N$.

For any $k\in \N$ and $f:\N \to\N$ there is $N'\in\N$ and $x\in D$ such that $N' \le \tilde{\phi}(24K((\mathring{f}+1)^{(R)}(0)+1)^2)$ with $\mathring{f}(m):=\max\{\tilde{f}(\tilde{\phi}(24K(m+1)^2)),24K(m+1)^2\}$ and $R$ as before, for which the following two properties hold:
\[\forall i<l \left(\|T_ix-x\| < \frac{1}{\tilde{f}(N')+1}\right)\]
\[\forall n\ge N' \left(\langle u-Ax,x_n-x \rangle < \frac{1}{k+1}\right)\]
In particular:
\[\forall n\in[N',f(N')] \left(\langle u-Ax,x_n-x \rangle < \frac{1}{k+1}\right)\]
\end{prop}
\begin{proof}
We apply Proposition \ref{quant-lem2.3} with $g(n):=\tilde{f}(\tilde{\phi}(n))$ in place of $f$. We get $N\in\N$ such that $N \le 24K((\mathring{f}+1)^{(R)}(0)+1)^2$ with $R:=12b_A K^2\left(\left\lceil\frac{1}{2}b_A\kone^2\right\rceil +2b_U\kone\right)\\(k+1)^2$ and $\mathring{f}(m):=\check{g}(m)=\max\{\tilde{f}(\tilde{\phi}(24K(m+1)^2)),24K(m+1)^2\}$ and we get $x\in D$ for which the following two properties hold:
\[\forall i<l \left(\|T_ix-x\| < \frac{1}{\tilde{f}(\tilde{\phi}(N))+1}\right)\]
\begin{equation*}
\forall y\in D \left(\forall i<l \left(\|T_iy-y\|\le\frac{1}{N+1}\right) \to \langle u-Ax,y-x \rangle < \frac{1}{k+1}\right)
\end{equation*}

We have that
\[\forall n\ge \tilde{\phi}(N)\forall i<l \left(\|T_ix_n-x_n\|\le\frac{1}{N+1}\right)\]
Since $\tilde{\phi}(N) \ge \nalpha$, by Lemma \ref{bound-lem} we get that $x_n\in D$ for $n\ge \tilde{\phi}(N)$, thus
\[\forall n\ge \tilde{\phi}(N) \left(\langle u-Ax,x_n-x \rangle < \frac{1}{k+1}\right)\]
Take $N':=\tilde{\phi}(N)$. Since $\tilde{\phi}$ is monotone, we have $N' \le \tilde{\phi}(24K((\mathring{f}+1)^{(R)}(0)+1)^2)$. Furthermore,
\[\forall i<l \left(\|T_ix-x\| < \frac{1}{\tilde{f}(N')+1}\right)\]
\end{proof}

\subsection{Rates of metastability}

The following two results correspond to the last step of Xu's proof.

\begin{prop} \label{step7}
Assume (C1q) and (C2q) hold. Let $k,n\in\N,x\in \overline{B}_K(p)$ and let $f:\N \to \N$ be a monotone function. If
\[\forall i<l \left(\|T_ix-x\| \le \frac{1}{\eta(k,n,f)+1}\right)\]
and
\[\forall j \in [n+\nalpha+1,f(\sigma(k,n))+\nalpha+1] \left(\langle u-Ax,x_j-x\rangle \le \frac{\gamma}{6(k+1)}\right)\]
with $\sigma(k,n)= \sigma_2 \left (\left \lceil \frac{n+\lceil \ln(12K^2(k+1))\rceil}{\gamma} \right\rceil +\nalpha+1\right)-\nalpha$, then
\[\forall j\in [\sigma(k,n)+\nalpha,f(\sigma(k,n)+\nalpha)]\left(\|x_j-x\|^2 \le \frac{1}{k+1}\right)\]
\end{prop}

\begin{proof}
For $n\ge \nalpha$, by \eqref{xn1-x-sq-ineq}:
\begin{eqnarray*}
\|x_{n+1}-x\|^2 &\le& (1-\alpha_{n+1}\gamma)\|x_n-x\|^2+2\|x_n-x\|\|T_{n+1}x-x\|\\
&&+\|T_{n+1}x-x\|^2+\alpha_{n+1}\gamma\frac{2}{\gamma}\langle u-Ax,x_{n+1}-x\rangle
\end{eqnarray*}
We show that we can apply Proposition \ref{XuMeta} with
$s_n:=\|x_{n+\nalpha}-x\|^2,\quad a_n:=\alpha_{n+\nalpha+1}\gamma,\quad b_n:=\frac{2}{\gamma}\langle u-Ax,x_{n+\nalpha+1}-x\rangle, \quad c_n:=2\|x_{n+\nalpha}-x\|\|T_{n+\nalpha+1}x-x\|+\|T_{n+\nalpha+1}x-x\|^2$ and $q:=f(\sigma(k,n))$ for arbitrary $k,n\in\N$.

Analogously to what was done in Claim 2 of Proposition \ref{tildeTn-as-reg}, we have that $\tilde{\Theta}(k) = \max\left\{\sigma_2 \left (\left \lceil \frac{k}{\gamma} \right\rceil +\nalpha+1\right)-\nalpha-1,0\right\}$ is a rate of divergence for $\sum\limits_{n=0}^{+\infty} a_n$.

Given $p\in F$, we have that $\|x_{n+\nalpha}-x\|^2 \le (\|x_{n+\nalpha}-p\|+\|p-x\|)^2 \le 4K^2$.

We have, for $j\in\N$: 
\begin{eqnarray*}
c_j &=&2\|x_{j+\nalpha}-x\|\|T_{j+\nalpha+1}x-x\|+\|T_{j+\nalpha+1}x-x\|^2\\
&\le& 4K\frac{1}{\eta(k,n,f)+1}+\frac{1}{(\eta(k,n,f)+1)^2}\\
&\le& 4K\frac{1}{24K(k+1)(f(\sigma(k,n))+1)+1}+\frac{1}{(\max\{f(\sigma(k,n)),6k+5\}+1)^2}\\
&\le& \frac{1}{6(k+1)(f(\sigma(k,n))+1)} +\frac{1}{6(k+1)(f(\sigma(k,n))+1)}\\
&=&\frac{1}{3(k+1)(f(\sigma(k,n))+1)}=\frac{1}{3(k+1)(q+1)}
\end{eqnarray*}

By hypothesis we get
\[\forall j \in [n,q] \left(b_j \le \frac{1}{3(k+1)}\right)\]
Thus, Proposition \ref{XuMeta} gives us
\[\forall j\in [\sigma(k,n),f(\sigma(k,n))]\left(\|x_{j+\nalpha}-x\|^2 \le \frac{1}{k+1}\right)\]
that is,
\[\forall j\in [\sigma(k,n)+\nalpha,f(\sigma(k,n))+\nalpha]\left(\|x_j-x\|^2 \le \frac{1}{k+1}\right)\]

with
\begin{align*}
\sigma(k,n)&=\theta(n+\lceil \ln(3L(k+1))\rceil)+1\\
&=\tilde{\Theta}(n+\lceil \ln(12K^2(k+1))\rceil)+1\\
&=\max\left\{\sigma_2 \left (\left \lceil \frac{n+\lceil \ln(12K^2(k+1))\rceil}{\gamma} \right\rceil +\nalpha+1\right)-\nalpha-1,0\right\}+1\\
&=\max\left\{\sigma_2 \left (\left \lceil \frac{n+\lceil \ln(12K^2(k+1))\rceil}{\gamma} \right\rceil +\nalpha+1\right)-\nalpha,1\right\}
\end{align*}
Note that, by Lemma \ref{div-rate-k-1}, we have that
\begin{equation*}
\sigma_2 \left (\left \lceil \frac{n+\lceil \ln(12K^2(k+1))\rceil}{\gamma} \right\rceil +\nalpha+1\right)
\ge\left \lceil \frac{n+\lceil \ln(12K^2(k+1))\rceil}{\gamma} \right\rceil +\nalpha
\ge\nalpha+1
\end{equation*}
thus
\[\sigma(k,n)=\sigma_2 \left (\left \lceil \frac{n+\lceil \ln(12K^2(k+1))\rceil}{\gamma} \right\rceil +\nalpha+1\right)-\nalpha\]

Now, given $g:\N\to\N$ monotone, take $f:\N\to\N$ with $f(j):=g(j+\nalpha)$ for all $j\in\N$. Then $f$ is monotone and $g(j+\nalpha) \le f(j)+\nalpha$ for all $j\in\N$. From the above we get
\[\forall j\in [\sigma(k,n)+\nalpha,g(\sigma(k,n)+\nalpha)]\left(\|x_j-x\|^2 \le \frac{1}{k+1}\right)\]
\end{proof}

\begin{prop} \label{step7-C2q**}
Assume (C1q) and (C2q*) hold. Let $k,n\in\N,x\in \overline{B}_K(p)$ and let $f:\N \to \N$ be a monotone function. If
\[\forall i<l \left(\|T_ix-x\| \le \frac{1}{\eta(k,n,f)+1}\right)\]
and
\[\forall j \in [n+\nalpha+1,f(\sigma^*(k,n))+\nalpha+1] \left(\langle u-Ax,x_j-x\rangle \le \frac{\gamma}{6(k+1)}\right)\]
with $\sigma^*(k,n)=\max\left\{\sigma_2^*\left(n, \frac{1}{12K^2(k+1)}\right)-\nalpha-l,1\right\}$, then
\[\forall j\in [\sigma^*(k,n)+\nalpha,f(\sigma^*(k,n)+\nalpha)]\left(\|x_j-x\|^2 \le \frac{1}{k+1}\right)\]
\end{prop}

\begin{proof}
We have \eqref{xn1-x-sq-ineq} for $n\ge \nalpha$. We show that we can apply Proposition \ref{XuMeta-C2q**} with
$s_n:=\|x_{n+\nalpha}-x\|^2,\quad a_n:=\alpha_{n+\nalpha+1}\gamma,\quad b_n:=\frac{2}{\gamma}\langle u-Ax,x_{n+\nalpha+1}-x\rangle, \quad c_n:=2\|x_{n+\nalpha}-x\|\|T_{n+\nalpha+1}x-x\|+\|T_{n+\nalpha+1}x-x\|^2$ and $q:=f(\sigma^*(k,n))$ for arbitrary $k,n\in\N$.

We know that, given $m\in\N$ and defining $\Theta^*$ as in \eqref{def-theta*}, $\Theta^*(m,\cdot)$ is a rate of convergence of $\prod\limits_{n=m}^\infty (1-a_n)$ to $0$.

We know from the proof of Proposition \ref{step7} that the conditions on $b_n$ and $c_n$ are satisfied.

Thus, Proposition \ref{XuMeta-C2q**} gives us
\[\forall j\in [\sigma^*(k,n),f(\sigma^*(k,n))]\left(\|x_{j+\nalpha}-x\|^2 \le \frac{1}{k+1}\right)\]
that is,
\[\forall j\in [\sigma^*(k,n)+\nalpha,f(\sigma^*(k,n))+\nalpha]\left(\|x_j-x\|^2 \le \frac{1}{k+1}\right)\]

with

\begin{align*}
\sigma^*(k,n)&={\rm A'}\left(n, \frac{1}{3D(k+1)}\right)+1\\
&=\Theta^*\left(n, \frac{1}{12K^2(k+1)}\right)+1\\
&=\max\left\{\sigma_2^*\left(n, \frac{1}{12K^2(k+1)}\right)-\nalpha-l-1,0\right\}+1\\
&=\max\left\{\sigma_2^*\left(n, \frac{1}{12K^2(k+1)}\right)-\nalpha-l,1\right\}.
\end{align*}

The proof concludes as in Proposition \ref{step7}.
\end{proof}


Below we show the conclusion of the proof for Theorems \ref{main} and \ref{main-C4q}:

\begin{proof}
We first apply Proposition \ref{step7} with $(k+1)^2-1$ in place of $k$ and $\tilde{f}$ in place of $f$. Thus, if

\[\forall i<l \left(\|T_ix-x\| < \frac{1}{\eta((k+1)^2-1,n,\tilde{f})+1}\right)\]
and
\[\forall j \in [n+\nalpha+1,\tilde{f}(\sigma((k+1)^2-1,n))+\nalpha+1] \left(\langle u-Ax,x_j-x\rangle \le \frac{\gamma}{6(k+1)^2}\right)\]

Then
\[\forall j\in [\sigma((k+1)^2-1,n)+\nalpha,\tilde{f}(\sigma((k+1)^2-1,n)+\nalpha)]\left(\|x_j-x\| \le \frac{1}{k+1}\right)\]

Thus, and also due to Proposition \ref{cond-innerprod-xn}, we may apply Proposition \ref{FLPres} with $X=H,D=\overline{B}_K(p),u_n=x_n,P=\nalpha+1,U_i=T_i,\phi_i(x,y)=\langle u-Ax, y\rangle$ and
\begin{eqnarray*}
\bar{\psi}(k,f)&=&\tilde{\phi}(24K((\mathring{f}+1)^{(R)}(0)+1)^2)\\
\bar{\delta}(k)&=&6\left\lceil\frac{1}{\gamma}\right\rceil(k+1)^2-1\\
\bar{\gamma}(k,n,f)&=&\eta((k+1)^2-1,n,\tilde{f})\\
\bar{\eta}(k,n,f)&=&\tilde{f}(\sigma((k+1)^2-1,n))\\
\bar{\sigma}(k,n,f)&=&\sigma((k+1)^2-1,n)+\nalpha
\end{eqnarray*}
then $\bar{\psi}(\bar{\delta}(2k+1),\bar{f})=\bar{\psi}(24\left\lceil\frac{1}{\gamma}\right\rceil(k+1)^2-1,\bar{f})=\tilde{\phi}(24K(h^{(\dd)}(0)+1)^2)$ where $\bar{f}(m) := \max\{\bar{\gamma}(2k+1,m,f),\bar{\eta}(2k+1,m,f)\}=\max\{\eta(4(k+1)^2-1,m,\tilde{f}),\\\tilde{f}(\sigma(4(k+1)^2-1,m))\}$.

\end{proof}

For Theorems \ref{main-C2q**} and \ref{main-C4q-C2q**}, we apply Proposition \ref{step7-C2q**} instead of Proposition \ref{step7} and replace $\sigma$ with $\sigma^*$.


\paragraph{Acknowledgements:} The author is grateful to his PhD advisors Fernando Ferreira and Lauren\c{t}iu Leu\c{s}tean for providing valuable comments and suggestions that improved this paper.

The author acknowledges the support of FCT -- Fundação para a Ciência e Tecnologia through a doctoral scholarship with reference number 2022.12585.BD as well as the support of the research center CEMS.UL under the FCT funding UIDB/04561/2025.

\end{document}